\documentclass{amsart}
\usepackage{amscd, amssymb, amsmath, amsthm}
\usepackage{amsfonts,enumerate}
\usepackage{latexsym,,txfonts}
\usepackage{graphics,graphicx,psfrag}
\usepackage[colorlinks=true]{hyperref}

\newtheorem{theorem}{Theorem}[section]
\newtheorem{lemma}[theorem]{Lemma}
\newtheorem{prop}[theorem]{Proposition}
\newtheorem{cor}[theorem]{Corollary}

\theoremstyle{definition}
\newtheorem{definition}[theorem]{Definition}
\newtheorem{example}[theorem]{Example}

\newtheorem{question}[theorem]{Question}

\theoremstyle{remark}
\newtheorem*{remark}{Remark}

\newcommand{\RR}{{\mathbb R}}
\newcommand{\QQ}{{\mathbb Q}}

\newcommand{\ZZ}{{\mathbb Z}}

\newcommand{\HH}{{\mathbb H}}
\newcommand{\EE}{{\mathbb E}}

\newcommand{\seq}[1]{{\{ {#1}\}}}
\newcommand{\norm}[1]{{\|{#1}\|}}

\newcommand{\SL}{{\mathrm{SL}}}

\newcommand{\Vol}{{\mathrm{Vol}}}
\newcommand{\Area}{{\mathrm{Area}}}
\newcommand{\diam}{{\mathrm{diam}}}

\newcommand{\geo}{{\mathtt{geo}}}

\newcommand{\tori}{{\mathcal{T}}}
\newcommand{\cpn}{{\mathtt{c}}}
\newcommand{\pat}{{\mathtt{p}}}

\begin{document}

\title{Presentation length and Simon's conjecture}

\author[Ian Agol]{%
        Ian Agol} 
\address{%
    University of California, Berkeley \\
    970 Evans Hall \#3840 \\
    Berkeley, CA 94720-3840} 
\email{% 
     ianagol@math.berkeley.edu}  
     
\author[Yi Liu]{%
        Yi Liu} 
\address{%
    University of California, Berkeley \\
    970 Evans Hall \#3840 \\
    Berkeley, CA 94720-3840} 
\email{% 
    yliu@math.berkeley.edu}  

\thanks{Agol and Liu partially supported by NSF grant DMS-0806027}

\date{% 
 July 6, 2010}

\begin{abstract}
In this paper, we show that any knot group maps onto at most finitely many knot groups.
This gives an affirmative answer to a conjecture of J. Simon. 
We also bound the diameter of a closed hyperbolic 3-manifold
linearly in terms of the presentation length of its fundamental group, improving a result of White. 
\end{abstract}

\maketitle

\tableofcontents

\section{Introduction}

In this paper, we prove the following theorem:

\begin{theorem}\label{main} Let $G$ be a finitely generated group with $b_1(G)=1$. 
Then there are at most finitely many distinct knot complements $M$ such that there is an epimorphism $\phi:G\twoheadrightarrow\pi_1(M)$.\end{theorem}

As a corollary, we resolve Problem 1.12(D) from Kirby's problem list \cite{Kirby},
conjectured by Jonathan K. Simon in the 1970's. Recall a {\it knot group} is the fundamental group of the complement of a knot in $S^3$. 
\begin{cor} Every knot group maps onto at most finitely many knot groups.\end{cor}
Note that this also resolves Problem 1.12(C). Our techniques say nothing about Part (B), and Part (A) is known to be false. 

There has been a fair amount of work recently on Simon's conjecture, which partly motivated our work. Daniel S.
Silver and Wilbur Whitten proved that a fibered knot group maps onto at most finitely many knot groups (\cite[the
comment after Conjecture 3.9]{SilverWhitten06}, see also \cite{Leininger09}).
Silver and Whitten also considered a restricted partial order on knot complements where the epimorphism must preserve
peripheral structure, and therefore has a well-defined degree, (\cite{SilverWhitten05, SilverWhitten08}). 
Restricting the class of epimorphisms to ones of 
non-zero degree, finiteness was shown by Michel Boileau, Hyam Rubinstein and Shicheng Wang \cite{BRW}.
Recently it was proven that a $2$-bridge knot group maps onto at most finitely 
many knot groups by Boileau et al. \cite{BBRW}. There
was also some experimental evidence for epimorphsims between prime  knot groups of $\leq11$ crossings by Teruaki Kitano, Masaaki Suzuki, et al. (\cite{HKMS, KitanoSuzuki08}). 
There are also families of examples of epimorphisms between 2-bridge knot groups, cf. \cite{ORS08, LS, HosteShanahan10}.

Now we give some remarks on the proof of Theorem \ref{main}. The finitely generated case is reduced to the finitely
presented case following a suggestion of Jack Button, so we assume that $G$ is finitely presented.  
Given the proof of Simon's conjecture for maps of non-zero degree \cite{BRW}, we need to allow maps of zero degree. 
In this case, and allowing finitely presented groups $G$ with $b_1(G)=1$, 
techniques such as simplicial volume give no information. Nevertheless, we can show that there is a bound on
the simplicial volume of the image in terms of the \emph{\hyperlink{term-prLen}{presentation length}} of $G$, defined by Daryl Cooper in \cite{Co} and used to bound the volume of a
hyperbolic manifold (Section \ref{Sec-volume}). The \hyperlink{term-prLen}{presentation length} gives a coarse substitute for simplicial volume. Suppose we want to prove Theorem \ref{main}
for epimorphisms to hyperbolic knot groups. Then bounding the volume does not give a finiteness result, since there can
be infinitely many hyperbolic knot complements of bounded volume, such as the twist knots, which are obtained by 
Dehn filling on the Whitehead link. A twist knot $k$ with a large number of twists will have a very short geodesic in the hyperbolic metric on its complement, 
which contains a very large tubular neighborhood by the Margulis Lemma. What we would like to do is
to factor the epimorphism $G \twoheadrightarrow G_k$ through the fundamental group of the Whitehead link complement obtained by drilling the
short geodesic from the twist knot complement. 
Intuitively, we expect that a \hyperlink{term-prLen}{presentation complex} for $G$ mapped into $S^3-k$ should be possible to homotope off
of a deep Margulis tube, since the complex should have bounded area, whereas the meridian disk of the Margulis tube should have large area. 
If we could do this, then we obtain a contradiction, since the only finitely generated covers of the Whithead link complement with $b_1=1$ are
elementary covers, and the homomorphism factoring through such a cover could not map onto $G_k$. 
However, this factorization cannot be done in general, cf. Example \ref{knotExample}. The substitute for this is to factor through the \emph{\hyperlink{term-extDrilling}{extended drilling}}, which has
enough good properties, such as coherence of the fundamental group (cf. Subsection \ref{Subsec-Scott}), 
that we may still obtain a contradiction (Section \ref{Sec-factor}). The general case of non-hyperbolic knots is similar,
but requires some modifications involving the JSJ decomposition of the knot complement and some case-by-case analysis (Sections \ref{Sec-cpnship}, \ref{Sec-number}, \ref{Subsec-SFPieces}, \ref{Sec-clTypes}). 
The technique for this factorization is based on a result of Matthew E. White (\cite{Wh}),
which bounds the diameter of a closed hyperbolic $3$-manifold in terms of the length of the presentation of its fundamental group. We improve upon White's result in Section \ref{Sec-diamBound} (to understand the improvement of White's result Theorem \ref{diamBound}(2), you need only read Section \ref{Subsec-DehnExt} and Section \ref{Subsec-factorHyp} ). 

{\bf Acknowledgement:} We thank Daniel Groves for helpful conversations. We thank  Jack Button, Hongbin Sun, and Shicheng Wang for comments on 
a preliminary draft, and the referee for their comments. 

\section{Dehn extensions}\label{Sec-eDehn}
In this section and the next, we study factorizations of maps through 
\hyperlink{term-extDehnFilling}{extended} \hyperlink{term-extDehnFilling}{Dehn} \hyperlink{term-extDehnFilling}{fillings}. 
This is motivated by the na\"ive question: suppose $f_n:K\to M_n$ is a sequence of maps from a finite $2$-complex to  orientable aspherical compact $3$-manifolds $M_n$, 
which are obtained by Dehn filling on a $3$-manifold $N$ along a sequence of slopes on a torus boundary component of $N$, 
does $f_n$ factorize as $i_n\circ f'$ up to homotopy for some $f':K\to N$ for infinitely many $n$, 
where $i_n:N\to M_n$ is the Dehn filling inclusion? The answer is negative as it stands, (cf. Subsection \ref{Subsec-example}), 
but a modified version will be true in certain natural situations if we allow `\hyperlink{term-extDehnFilling}{extended Dehn fillings}' 
(cf. Theorems \ref{factor-hyp}, \ref{factor-SF}). 
We introduce and study \hyperlink{term-DehnExt}{Dehn extensions} in Subsections \ref{Subsec-DehnExt}, \ref{Subsec-Scott}.

\subsection{Examples}\label{Subsec-example}
We start with an elementary example. The phenomenon here essentially illustrates why 
we should expect the factorization only through the 
\hyperlink{term-extDehnFilling}{extended Dehn filling}, and will be used in the construction of Dehn extensions.

\begin{example}\label{elemExample} Let $P=\ZZ\alpha\oplus\ZZ\beta$ be a rank-$2$ free-abelian group generated by $\alpha,\beta$. We
identify $P$ as the integral lattice of $\RR^2$, where $\alpha=(1,0)$, $\beta=(0,1)$. Let
$\omega=a\,\alpha+b\,\beta\in P$ be primitive and $m>1$ be an integer. Define an extended lattice of $\RR^2$ 
by $\tilde{P}=P+\ZZ\,\frac{\omega}{m}$. Pick any primitive $\zeta\in P$ 
 such that $\omega+\zeta\in m\ZZ\oplus m\ZZ$. There are infinitely many such $\zeta$'s, 
 for example, $\zeta=(my-a)\,\alpha+(-mx-b)\,\beta$ with integers pairs $(x,y)$ such that $ax+by=1$. 
 Note there are infinitely many such pairs since the general solution to the linear Diophantine equation $ax+by=1$ is $(x,y)=(x^*+bt,y^*-at)$, 
 where $t\in\ZZ$ and $(x^*,y^*)$ is a particular solution. 
 $\zeta$ is primitive because $(-x)(my-a)+(-y)(-mx-b)=1$. Thus there is a well-defined epimorphism, $$\phi_\zeta:\tilde{P}\to P\,/\,\ZZ\zeta,$$
by requiring $\phi_\zeta|_P$ to be modulo $\ZZ\zeta$, and $\phi_\zeta(\frac{\omega}{m})=\frac{\omega+\zeta}{m}+\ZZ\zeta$.
%Then there is no homomorphism $r:\tilde{P} \to P$ such that  $\phi_\zeta = \iota_\zeta \circ r$ for all $\zeta$, where $\iota_\zeta$ is  the `Dehn
%filling' along $\zeta$, namely the canonical quotient $\iota_\zeta:P\to P\,/\,\ZZ\zeta$. 
%Note of course $\phi_\zeta$ factors through
%$\iota^e_\zeta:P+\ZZ\frac{\zeta}{m}\to (P+\ZZ\frac{\zeta}{m})\,/\,\ZZ\frac{\zeta}{m}\cong P\,/\,\ZZ\zeta$, since $\tilde{P}=P+\ZZ\frac{\zeta}{m}$.
\end{example}

We extend this construction to give a counterexample to the question in the category of $3$-manifolds.

\begin{example}\label{knotExample} Let $k_{p/q}$ be the $p/q$-cable knot ($p,q$ coprime and $q>1$) on a hyperbolic knot $k\subset S^3$. Let $N=S^3-k_{p/q}$, $M=S^3-k$. Identify $P=\pi_1(\partial M)=\ZZ\mu\oplus\ZZ\lambda$ as a peripheral subgroup of $\pi_1(M)$, where $\mu$, $\lambda$ are the meridian and the longitude. Let $w=\mu^p\lambda^q$, then $\pi_1(N)\cong \pi_1(M)\langle\sqrt[q]{w}\rangle$, by which we mean the group $\pi_1(M)*\langle u\rangle$ modulo $u^q=w$. We may as well write $\pi_1(N)\cong \pi_1(M)*_PQ$, where $Q=P\langle\sqrt[q]{w}\rangle$. Note 
the abelianization of $Q$ equals $\tilde{P}=P+\ZZ\,\frac{\omega}{q}$, where $\omega=p\,\mu+q\,\lambda$, if we identify $P=\ZZ\mu\oplus\ZZ\lambda$ as the integral lattice of $\QQ\oplus\QQ$. By Example \ref{elemExample}, there are infinitely many primitive $\zeta$'s such that $\omega+\zeta\in q\ZZ\oplus q\ZZ$, and there are epimorphisms $\bar\phi|:\tilde{P}\to P\,/\,\ZZ\zeta$, which extend as 
$\bar\phi:\pi_1(M)*_P\tilde{P}\to\pi_1(M_\zeta)$ in an obvious fashion, where $M_\zeta$ is the Dehn filling along $\zeta\in P$. Hence we obtain epimorphisms $\phi:\pi_1(N)\to\pi_1(M_\zeta)$ via the composition:
$$\pi_1(N)\cong\pi_1(M)*_PQ\to\pi_1(M)*_P\tilde{P}\to\pi_1(M_\zeta).$$
Now there are infinitely many slopes $\zeta\subset\partial M$, such that the corresponding $\phi:\pi_1(N)\to \pi_1(M_\zeta)$ are all surjective. It also is clear that $\phi$ can be realized by a map $$f:N\to M_\zeta,$$ by mapping the companion \hyperlink{term-piece}{piece} of $N$ to $M_\zeta$, and extending over the cable pattern \hyperlink{term-piece}{piece} of $N$.

However, $f$ does not always factor through the Dehn filling $i:M\to M_\zeta$ up to homotopy. 
Suppose $f\simeq i\circ h$ for some $h:N\to M$. 
Let $\kappa:\tilde{M}\to M$ be the covering corresponding to ${\rm Im}(h_\sharp)$, where $h_\sharp:\pi_1(N)\to\pi_1(M)$ after choosing some base points. Let $j:M\to N$ be the inclusion of $M$ as the companion \hyperlink{term-piece}{piece} of $N$, and let $\kappa':\tilde{M}'\to \tilde{M}$ be the covering corresponding to ${\rm Im}((h\circ j)_\sharp)$. 
The homotopy lift $\tilde{h}:M \to\tilde{M}$ is hence $\pi_1$-surjective, and the homotopy lift $\widetilde{h\circ j}: M\to \tilde{M}'$ is also $\pi_1$-surjective. We have the commutative diagram:
$$\begin{CD} M @>j>> N @.\\
@V \widetilde{h\circ j} V V @V \tilde{h} VV \\
\tilde{M}'@>\kappa' >> \tilde{M} @>\kappa>> M.
\end{CD}$$
Note $H_1(N)\cong H_1(M) \cong \ZZ$ and $\tilde{M}$ and $\tilde{M}'$ are hyperbolic with non-elementary fundamental group, since they map $\pi_1$-surjectively to $M_\zeta$ via the factorization. Clearly $H_1(\tilde{M})$ and $H_1(\tilde{M}')$ are isomorphic to $\ZZ$.
As we shall see in Theorem \ref{volBound}, the volume of $\tilde{M}$ is at most $\pi\ell(\pi_1(N))$, where $\ell(\pi_1(N))$ is the \hyperlink{term-prLen}{presentation length} of $\pi_1(N)$, and similarly the volume of $\tilde{M}'$ is at most $\pi\ell(\pi_1(M))$. Thus $\tilde{M}$ is a finite covering of $M$ of some degree $d\leq\pi\ell(\pi_1(M))/\Vol(M)$. Moreover, $\tilde{h}_*: H_1(N) \to H_1(\tilde{M})$ is an isomorphism as $\tilde{h}$ is $\pi_1$-surjective, and $\kappa_*:H_1(\tilde{M})\to H_1(M)$ is the multiplication by some factor $d_0$ of $d$. Similarly, $(\widetilde{h\circ j})_*: H_1(M) \to H_1(\tilde{M}')$ is an isomorphism, and $\kappa'_*:H_1(\tilde{M}')\to H_1(\tilde{M})$ is multiplication by a factor $q'$, which is the same factor as the multiplication $j_*: H_1(M)\to H_1(N)$. By the cabling construction of knots, the image of $j_*:H_1(M)\to H_1(N)\cong\ZZ$ is contained in $q\ZZ$, and thus we see that $q'=q$. Thus:
$$\begin{CD} H_1(M) @>\times q>> H_1(N) @.\\
@V \cong V V @V \cong VV \\
H_1(\tilde{M}')@>\times q>> H_1(\tilde{M}) @>\times d_0>> H_1(M).
\end{CD}$$
However, we have $q= \Vol(\tilde{M}')\,/\,\Vol(\tilde{M}) \leq \pi\ell(\pi_1(M))\,/\,\Vol(M)$. Therefore if
we take $q > \pi \ell(\pi_1(M))\,/\,\Vol(M)$, we obtain a contradiction. In other words, for such $q$'s, $f$ does not  factor through the Dehn filling $i:M\to M_\zeta$ up to homotopy. 

On the other hand, as evidence for Theorem \ref{factor-hyp}, clearly every $\phi:\pi_1(N)\to\pi_1(M_\zeta)$ as above factors through the `extended' Dehn filling epimorphisms $\iota^e:\pi_1(M)*_P(P+\ZZ\frac{\zeta}{q})\to \pi_1(M_\zeta)$, since $P+\ZZ\frac{\zeta}{q}=\tilde{P}$.\end{example}

\subsection{Dehn extensions of filling and drilling}\label{Subsec-DehnExt}
In this subsection, we introduce the notion of \hyperlink{term-DehnExt}{Dehn extensions}.

Let $N$ be an aspherical orientable compact $3$-manifold, and $\zeta$ be a slope on an incompressible torus boundary component $T\subset\partial N$. By choosing a base-point of $N$ and a path to $T$, we may identify $P=\pi_1(T)$ as a peripheral subgroup of $\pi_1(N)$. By choosing an orientation of $\zeta$, we identify $\zeta$ as primitive element in $P$. On the other hand, by choosing a basis for $\pi_1(T)$, we may also identify $P\cong\ZZ\oplus\ZZ$ as the integral lattice in $\QQ\oplus\QQ$. 

\begin{definition} \raisebox{\baselineskip}[0pt]{\hypertarget{term-DehnExt}}For any integer $m>1$, we define the \emph{Dehn extension} of $\pi_1(N)$ 
along $\zeta$ with denominator $m$ as the amalgamated product: 
$$\pi_1(N)^{e(\zeta,m)}=\pi_1(N)*_P \left(P+\ZZ\,\frac{\zeta}{m}\right).$$
We often simply abbreviate this $\pi_1(N)^e$ when $m$ and $\zeta$ are clear from the context. There is a natural \emph{extended Dehn filling epimorphism}: $$\iota^e:\pi_1(N)^e\to\pi_1(N_\zeta),$$ defined by quotienting out the normal closure of $\ZZ\,\frac{\zeta}{m}$, where $N_\zeta$ is the Dehn filling of $N$ along $\zeta$.
We also regard $\pi_1(N)$ as a \emph{trivial} Dehn extension of itself with denominator $m=1$, in which case we do not require $N$ to have any incompressible torus boundary component or non-empty boundary.\end{definition}

From a topological point of view, the \hyperlink{term-DehnExt}{Dehn extension} $\pi_1(N)^e$ may be regarded as the fundamental group of a topological space $N^e=N^{e(\zeta,m)}$, which will be called the \emph{Dehn extension} of $N$ along $\zeta$ with denominator $m$. Certainly $N^e$ could be $N$ itself for the trivial \hyperlink{term-DehnExt}{Dehn extension}. If $m>1$, $N^e$ is obtained from $N$ by gluing $T\subset \partial N$ to the source torus of a mapping cylinder, denoted as $Z$, of the $m$-fold covering map between tori: $$\RR^2/P\to \RR^2/(P+\ZZ\frac{\zeta}{m}).$$
To visualize this, take the product of the unit interval $I=[0,1]$ with an $m$-pod (i.e. a cone over $m$ points), identify the $0$-slice with the $1$-slice via a primitive cyclic permutation of the legs, and denote the resulting space as $\Psi=\Psi(m)$. \raisebox{\baselineskip}[0pt]{\hypertarget{term-ridge}}The `side' of $\Psi$ is a loop which is homotopic to $m$ wraps along of the `ridge' loop of $\Psi$. The mapping cylinder $Z=Z(m)$ is homeomorphic to $S^1\times \Psi$, and we will refer the product of $S^1$ with the side loop (resp. ridge loop) as the \emph{side torus} (resp. \emph{ridge torus}). Glue the side torus of $Z$ to the boundary $T$ via a homeomorphism, such that the side-loop of $Z$ is identified with $\zeta$, (clearly this determines the resulting space up to homeomorphism), cf. Figure \ref{figDehnExt}. We call $Z$ the \emph{ridge piece}, and the copy of $N$ under the inclusion the \emph{regular part}.

\begin{figure}[htb]
\centering
\psfrag{Y}{\scriptsize{$N$}}
\psfrag{g}{\scriptsize{$\zeta$}}
\psfrag{a}[][]{\scriptsize{\shortstack{glue up with a\\ rotation of $\frac{2\pi}{m}$}}}
\psfrag{s}{\scriptsize{the side-loop}}
\psfrag{x}{\scriptsize{$\times$}}
\psfrag{P}{\scriptsize{$\Psi$}}
\psfrag{Z}{\scriptsize{$Z$}}
\psfrag{S}{\scriptsize{$S^1$}}
\psfrag{B}{\scriptsize{$T$}}
\psfrag{i}[][]{\scriptsize{identify $\partial$'s}}
\includegraphics[scale=.7]{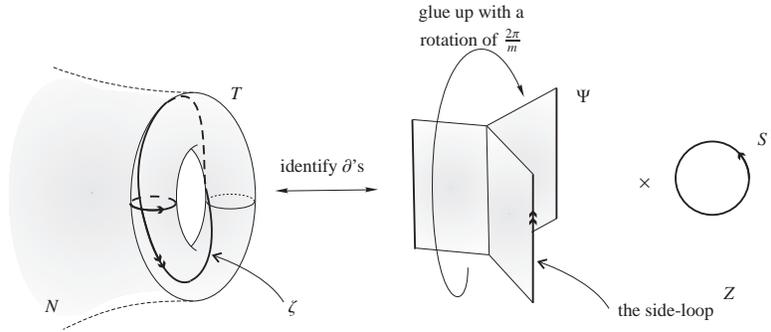}
\caption{The  Dehn extension $N^{e}$ along $\zeta$, illustrated with $m=3$.}\label{figDehnExt}
\end{figure}

The \hyperlink{term-DehnExt}{Dehn extension} $N^e$ may be \textit{ad hoc} characterized as a `ridged manifold'.
Despite the mild singularity near the ridges, $N^e$ behaves like an aspherical compact $3$-manifold in many ways. For instance, $N^e$ is an Eilenberg-MacLane space of $\pi_1(N)^e$, and $H_*(N^e,\QQ)\cong H_*(N,\QQ)$.
The \hyperlink{term-DehnExt}{Dehn extension} $N^e$ also admits an analogous JSJ decomposition. \raisebox{\baselineskip}[0pt]{\hypertarget{term-piece}}Recall that the classical Jaco-Shalen-Johanson (JSJ) decomposition says any orientable irreducible compact $3$-manifold can be cut along a minimal finite collection of essential tori into atoroidal and Seifert fibered {\it pieces}, canonical up to isotopy. The analogous JSJ decomposition of $N^e$ consists of the JSJ pieces of $N$ together with a \hyperlink{term-ridge}{ridge piece} $Z$
if $m>1$. These pieces are glued along $\pi_1$-injective tori. This is a graph-of-spaces decomposition, associating to $N^e$ a finite connected graph, called the \emph{JSJ graph}, whose vertices correspond to the JSJ pieces and edges to the JSJ tori.
Moreover, $\pi_1(N^e)=\pi_1(N)^e$ is coherent; indeed, any covering of $N^e$ with finitely
generated fundamental group has a Scott core, as we shall see in Subsection \ref{Subsec-Scott}.

There is also a natural map, called the \raisebox{\baselineskip}[0pt]{\hypertarget{term-extDehnFilling}}\emph{extended Dehn filling} of $N^e$ along $\zeta$, 
$$i^e:N^e\to M,$$ defined by collapsing $Z\cong S^1\times \Psi$ to $S^1$. $i^e$ induces
$\iota^e:\pi_1(N)^e\cong\pi_1(N^e)\to\pi_1(N_\zeta)$.

To an essential drilling, we may naturally associate a \hyperlink{term-DehnExt}{Dehn extension}. Specifically, let $M$ be an aspherical orientable compact $3$-manifold, and $\gamma\subset M$ be an essential simple closed curve, i.e. which is not null-homotopic. Then $N=M-\gamma$ is an aspherical orientable compact $3$-manifold, and the boundary component $\partial_\gamma N$ coming from the drilling is an incompressible torus.There is a canonical slope $\zeta\subset\partial_\gamma N$ so that $M=N_\zeta$.
\raisebox{\baselineskip}[0pt]{\hypertarget{term-extDrilling}} For any integer $m>0$, we shall call the \hyperlink{term-DehnExt}{Dehn extension} $N^e(\zeta, m)$ the \emph{extended drilling} of $M$ along $\gamma$ with denominator $m$.

We shall be most interested in \emph{geometric drillings}, namely, when $\gamma$ is a 
simple closed geodesic in the interior of some geometric \hyperlink{term-piece}{piece} of $M$. Factorization of 
maps through extended geometric drillings will be studied in
Section \ref{Sec-factor}.

\subsection{Coherence and the Scott core} \label{Subsec-Scott}

With notation from Subsection \ref{Subsec-DehnExt}, in this subsection, we show 
every \hyperlink{term-DehnExt}{Dehn extension} $N^e$ is \emph{coherent}, in the 
sense that every finitely generated subgroup of its fundamental group is finitely 
presented. This follows from the existence of a Scott core for any connected finite 
generated covering space $\tilde{N}^e$ of $N$ proved in Proposition \ref{ScottCore} below. 
The results in this section are preparatory for the arguments in Section \ref{Sec-homeoTypes}.
Recall for a topological space $X$, a \emph{Scott core} of $X$ is a connected compact 
subspace $C\subset\tilde{N}^e$, if any, whose inclusion induces a $\pi_1$-isomorphism. 
It is named after Peter Scott who first found such cores for  $3$-manifolds with finitely generated fundamental group (\cite{Sc}) . 
First, we need an auxiliary result. For background on combinatorial group theory, cf. \cite{Cohen}. 

\begin{lemma} \label{amalgamation}
Let $G=A\ast_C B$ or $G= A\ast_C$ be a free product with amalgamation or HNN extension. If $G$ and $C$ are finitely generated, then $A$ and $B$ are too. 
\end{lemma}
\begin{proof}
We give the argument for the amalgamated product, the HNN case being similar. 
Let $g_1,\ldots, g_n$ be generators for $G$, and $c_1,\ldots,c_m$ be generators for $C$. 
We may write each element $g_i = a_{i,1} b_{i,1} \cdots a_{i,k(i)} b_{i,k(i)}$, where
$a_{i,j}\in A, b_{i,j}\in B$, $i=1,\ldots, n$. Let $A' = \langle a_{i,j}, c_l  \rangle, B' =\langle a_{i,j}, c_l\rangle$, so 
that $C< A', C<B'$. 
Then we have $G'= A'\ast_C B' \to \langle A',B'\rangle \leq   A\ast_C B = G$ injects by \cite[Proposition 28]{Cohen}. But clearly also $G\leq G'$,
so $G=G', A'=A, B'=B$, and we see that $A$ and $B$ are finitely generated. 
\end{proof}

\begin{cor} \label{coherent}
If $G$ is a finitely generated group which is a finite graph of groups with finitely generated edge groups,
then the vertex groups are also finitely generated. 
\end{cor}
\begin{proof}
Use Lemma \ref{amalgamation} together with induction on the number of vertices of the graph of groups. 
\end{proof}

\begin{prop}\label{ScottCore} Let $N$ be an orientable aspherical compact $3$-manifold, and $N^e$ be the \hyperlink{term-DehnExt}{Dehn extension}
of $N$ with denominator $m>0$ along some slope on an incompressible torus boundary component. 
Then any  connected covering space $\kappa:\tilde{N}^e\to N^e$ with finitely generated fundamental 
group has an aspherical Scott core $C\subset\tilde{N}^e$. Furthermore,  for any component of the preimage 
of a JSJ torus of $N^e$, $C$ either meets it in a Scott core of the component or misses it. For 
 any component of the preimage of a JSJ \hyperlink{term-piece}{piece} of $N^e$, 
$C$ either meets it in an aspherical Scott core of the component or misses it.
\end{prop}

\begin{remark} The construction here includes the case when $N^e$ is the 
trivial \hyperlink{term-DehnExt}{Dehn} \hyperlink{term-DehnExt}{extension}, and so will Lemmas \ref{trivialCk} and \ref{simplifyCore}.\end{remark}

\begin{proof}
Let $\kappa:\tilde{N}^e\to N^e$ be the covering map, and let $\tori$ be the union of the JSJ tori of $N^e$. 
Pick a finite bouquet of circles $L=S^1\vee\cdots\vee S^1$ mapping $\pi_1$-surjectively into $\tilde{N}^e$, 
say $f:L\to\tilde{N}^e$. There are only finitely many components of the preimage of JSJ tori meeting $f(L)$, 
and there are only finitely many components of the preimage of \hyperlink{term-piece}{pieces} meeting $f(L)$.

Let $T\subset\tori$ be a JSJ torus. For any component $\tilde{T}$ of $\kappa^{-1}(T)$ meeting $f(L)$, 
clearly $\tilde{T}$ is either a torus, or a cylinder, or a plane, so we may take the torus, or an essential annulus, 
or a disk in $\tilde{T}$ containing $\tilde{T}\cap f(L)$, respectively, and denote as $C_{\tilde T}$. 
Let $C_T$ be the union of $C_{\tilde{T}}$ for all such $\tilde{T}$'s, and let $C_{\tori}$ be the union of all $C_T$'s.

For a regular \hyperlink{term-piece}{piece} $J$, first observe any component $\tilde{J}$ of $\kappa^{-1}(J)$ has finitely generated fundamental group. 
To see this, we may assume $\tilde{J}$ meets $f(L)$. Then there are at most finitely many components of $\tilde{N}^e-\tilde{J}$ 
which meet $f(L)$, say $U_1,\cdots,U_t$, and $\tilde{J}\cup U_1\cup\cdots\cup U_t\subset \tilde{N}^e$ is 
$\pi_1$-surjective as it contains $f(L)$. Clearly it is also $\pi_1$-injective, so $\pi_1(\tilde{N}^e)\cong\pi_1(\tilde{J}\cup U_1\cup\cdots\cup U_t)$. 
Note $\tilde{J}$ and $U_i$'s are glued along $\pi_1$-injective coverings of a JSJ torus, namely planes, cylinders or tori, 
so we have a graph-of-groups decomposition of $\pi_1(\tilde{N}^e)$, 
whose vertex groups are $\pi_1(\tilde{J}),\pi_1(U_1),\cdots,\pi_1(U_t)$ and whose edge groups are finitely generated. 
By Corollary \ref{coherent}, it follows the vertex groups must all be finitely generated as $\pi_1(\tilde{N}^e)$ is finitely generated. 
In particular, $\pi_1(\tilde{J})$ is finitely generated.

Thus, as $\tilde{J}\cap C_{\tori} \subset \partial\tilde{J}$ is a (possibly empty) compact sub-manifold, 
by a theorem of Darryl McCullough \cite[Theorem 2]{McC86}, there is a Scott core of $\tilde{J}$ which meets 
$\partial\tilde{J}$ exactly in $\tilde{J}\cap C_{\tori}$. Moreover, since $\tilde{J}$ is aspherical, 
we may also make the core aspherical by adding to the core bounded complementary components 
whose inclusion into $\tilde{J}$ is $\pi_1$-trivial. The result is denoted as $C_{\tilde{J}}$. 
Let $C_J$ be the union of all such $\tilde{J}$'s.

Let $Z$ be the ridge \hyperlink{term-piece}{piece} of $N^e$.
\raisebox{\baselineskip}[0pt]{\hypertarget{term-ridgeCovering}} For each component $\tilde{Z}$ of $\kappa^{-1}(Z)$ meeting $f(L)$, 
the preimage of the \hyperlink{term-ridge}{ridge torus} is either a torus, or a cylinder, or a plane, 
so we may take the torus, or an essential annulus, or a disk in the image, respectively, 
and thicken it up to a regular neighborhood in $\tilde{Z}$ which meets the preimage of the 
\hyperlink{term-ridge}{side torus} exactly in $\tilde{Z}\cap C_{\tori}$. 
This is a Scott core $C_{\tilde{Z}}$ of $\tilde{Z}$. Let $C_Z$ be the union of all such $\tilde{Z}$'s.

Let $C$ be the union of all the $C_J$'s and $C_Z$ constructed above. Then $C$ is an aspherical Scott core of $\tilde{N}^e$ as required.
\end{proof}

\raisebox{\baselineskip}[0pt]{\hypertarget{term-chunk}}In this paper, we often refer to any connected union $Q$ of a few components from $C_Z$ and $C_J$'s as a \emph{chunk} of $C$, following the notation in the proof of Proposition \ref{ScottCore}. It is a \emph{ridge} chunk (resp. a \emph{hyperbolic} chunk, or a \emph{Seifert fibered} chunk) if it is a single component from $C_Z$ (resp. some $C_J$ where $J$ is a hyperbolic \hyperlink{term-piece}{piece}, or a Seifert fibered \hyperlink{term-piece}{piece}). It is a \emph{regular} chunk if it is contained in $C-C_Z$. Note if the denominator of the \hyperlink{term-DehnExt}{Dehn extension} $N^e$ is $2$, there could be some ridge chunk which is a manifold, homeomorphic to either $I\times D^2$, or $S^1\times A$, or $I\times A$, or $S^1\times R$, or $I\times R$, where $A$ is an annulus and $R$ is a M\"obius strip, but we do not regard such as a regular chunk. In particular, regular chunks are all orientable.
The \raisebox{\baselineskip}[0pt]{\hypertarget{term-cutBdry}} \emph{cut boundary} $\partial_{\tori}Q$ of a chunk $Q$ is the union of the components of $C_{\tori}$ which are contained in the frontier of $Q$ in $\tilde{N}^e$. For example, if $Q$ is regular, $\partial_{\tori}Q$ is a (possibly disconnected, with boundary) essential compact subsurface of $\partial Q$.

The preimage of the JSJ tori $\tori$ of $N^e$ cuts $C$ into minimal \hyperlink{term-chunk}{chunks} along $C_\tori=C\cap\kappa^{-1}(\tori)$. This induces a graph-of-spaces decomposition of $C$ over a finite connected simplical graph $\Lambda$. The vertices of $\Lambda$ correspond to the components of $C-C_\tori$, and whenever two components are adjacent to each other, there are edges joining them, each corresponding to a distinct component of $C_\tori$ along which they are adjacent. Thus a \hyperlink{term-chunk}{chunk} $Q$ may be regarded as the subspace of $C$ associated to a connected complete subgraph of $\Lambda$. The graph-of-spaces decomposition also induces a graph-of-groups decomposition of $\pi_1(C)$ along free abelian edge groups of rank at most $2$.

In the rest of this subsection, we discuss how to get rid of unnecessary \hyperlink{term-chunk}{chunks} when $b_1(\tilde{N}^e)=1$.

\begin{lemma}\label{trivialCk} A \hyperlink{term-chunk}{chunk} $Q\subset C$ is contractible if and only if $b_1(Q)=0$.\end{lemma}

\begin{proof}
It suffices to prove the `if' direction. This is clear if $Q$ is a ridge \hyperlink{term-chunk}{chunk}. Also, if $Q$ a regular \hyperlink{term-chunk}{chunk}, $\partial Q$ must be a union of spheres so $Q$ is contractible as it is aspherical. To see the general case, first observe any component of $C_{\tori}\cap Q$ is separating in $Q$  as $b_1(Q)=0$. If there were a ridge \hyperlink{term-chunk}{subchunk} $S\subset Q$ with a \hyperlink{term-ridgeCovering}{ridge annulus or torus}, then any regular \hyperlink{term-chunk}{chunk} $R$ adjacent to $S$ would have at least one component of $\partial R$ which is not a sphere since it contains an essential loop, so $b_1(S)\geq1$. The union $Q'$ of $S$ and all its adjacent maximal regular \hyperlink{term-chunk}{chunks} has $b_1(Q')\geq 1$ by an easy Mayer-Vietoris argument. Furthermore, if $S'$ is a ridge \hyperlink{term-chunk}{chunk} adjacent to $Q'$, let $Q''$ is the union of $Q'$, $S'$ and all other maximal regular \hyperlink{term-chunk}{chunks} adjacent to $S'$. No matter $S'$ has a \hyperlink{term-ridgeCovering}{ridge disk, annulus or torus}, a Mayer-Vietoris argument again shows $b_1(Q'')\geq 1$. Keep going in this way, then in the end we would find $b_1(Q)\geq 1$.
Since $b_1(Q)=0$ by assumption, we see every ridge \hyperlink{term-chunk}{chunk} has a \hyperlink{term-ridgeCovering}{disk ridge}, hence is contractible. However, in this case, $\partial_{\tori}S$ is a union of disks for any ridge \hyperlink{term-chunk}{chunk} $S$, so for any maximal regular \hyperlink{term-piece}{piece} $R$ we must have $b_1(R)=0$ under the assumption $b_1(Q)=0$. This implies $Q$ is also contractible as we mentioned at the beginning. Therefore, $Q$ must be a contractible \hyperlink{term-chunk}{chunk}.
\end{proof}

For any component $\tilde{J}\subset \tilde{N}^e$ of the preimage of a JSJ \hyperlink{term-piece}{piece} $J\subset N^e$, the corresponding \hyperlink{term-chunk}{chunk} $C_{\tilde{J}}=C\cap \tilde{J}$ is called \emph{elementary} if $\pi_1(\tilde{J})$ is abelian, and it is called \emph{central} if $\pi_1(\tilde{J})$ is a subgroup of the center of $\pi_1(J)$. A central hyperbolic \hyperlink{term-chunk}{chunk} is always contractible, and a central Seifert fibered \hyperlink{term-chunk}{chunk} is either contractible or homeomorphic to a trivial $S^1$-bundle over a disk, and every ridge \hyperlink{term-chunk}{chunk} is central.

\begin{lemma}\label{simplifyCore} If $\pi_1(\tilde{N}^e)$ is non-abelian and $b_1(\tilde{N}^e)=1$, one may assume $C$ has no contractible \hyperlink{term-chunk}{chunk}, no non-central elementary \hyperlink{term-chunk}{chunk}, and $C_{\tori}=C\cap\kappa^{-1}(\tori)$ has no disk component.\end{lemma}

\begin{proof} We show this by modifying $C$. If there is a contractible \hyperlink{term-chunk}{chunk} $Q\subset C$, 
the \hyperlink{term-cutBdry}{cut} \hyperlink{term-cutBdry}{boundary} $\partial_{\tori}Q$ is a disjoint union of disks. 
If $D\subset \partial_{\tori}Q$ were non-separating  in $C$, $C'=C-D$ will have $b_1(C')=0$, 
then by Lemma \ref{trivialCk}, $C'$ is contractible, so $C$ is homotopy equivalent to a circle, 
which is impossible as $\pi_1(C)=\pi_1(\tilde{N}^e)$ is non-abelian by assumption. 
Thus $D\subset \partial_{\tori}Q$ is separating in $C$, so $C-D=C'\sqcup C''$, 
and $b_1(C')=0$ or $b_1(C'')=0$. Say $b_1(C')=0$, by Lemma \ref{trivialCk}, 
$C'$ is contractible. We may discard a small regular neighborhood of $C'$ from $C$, 
and the rest of $C''$ is still a Scott core. 
As there are at most finitely many \hyperlink{term-chunk}{chunks} in $C$, we may discard all the contractible ones, and obtain a Scott core, still denoted as $C$, with no contractible \hyperlink{term-chunk}{chunk}. 
A similar argument shows one may assume $C_{\tori}$ has no disk component. 
Now for any non-central elementary \hyperlink{term-chunk}{chunk} $Q$, 
$Q\neq C$ since $\pi_1(C)$ is non-abelian. As $C$ has no contractible \hyperlink{term-chunk}{chunk}, 
at least one component $U$ of $\partial_{\tori}Q$ is a torus or an annulus. 
Note when $Q$ is a hyperbolic \hyperlink{term-chunk}{chunk}, it is homeomorphic to $U\times I$, 
and  when $Q$ is a Seifert fibered \hyperlink{term-chunk}{chunk}, 
it is either a trivial $S^1$-bundle over an annulus if $U$ is a torus, 
or an $I$-bundle over an annulus if $U$ is an annulus, since $Q$ is non-central. 
Let $\tilde{J}$ be the component of the preimage of a regular JSJ \hyperlink{term-piece}{piece} $J$ 
such that $U$ is carried by a component $\tilde{T}\subset\partial\tilde{J}$. 
Suppose there were some other boundary component $U'$ carried by some 
other component $\tilde{T}'\subset\partial\tilde{J}$. As $U$, 
$U'$ are both essential annuli or tori by the construction of Proposition \ref{ScottCore}, 
$T$, $T'$ cannot be disjoint cylinders or tori, otherwise $\tilde{J}$ would not be elementary.
We conclude $T=T'$, hence $U=U'$ by the construction. 
It is also clear that $\pi_1(U)\cong\pi_1(Q)\cong\pi_1(\tilde{J})$ induced by the inclusion. 
We may discard a small regular neighborhood of $Q$ from $C$, 
and the rest is still connected and is a Scott core. 
Thus we discard all non-central elementary \hyperlink{term-chunk}{chunks} in this fashion. 
In the end, we may assume $C$ to have no non-central elementary \hyperlink{term-chunk}{chunk}. 
Note also the modifications here do not affect the properties of $C$ described in Proposition \ref{ScottCore}, 
so we obtain a Scott core $C$ as required.
\end{proof}

\section{Factorization through extended geometric drilling} \label{Sec-factor}
In this section, we show that a map from a given finite $2$-complex to an orientable aspherical compact $3$-manifold factorizes up to homotopy through \hyperlink{term-extDrilling}{extended} \hyperlink{term-extDrilling}{drilling} along sufficiently short geodesics in the hyperbolic \hyperlink{term-piece}{pieces} (Theorem \ref{factor-hyp}) and through \hyperlink{term-extDrilling}{extended} \hyperlink{term-extDrilling}{drillings} along a sufficiently sharp cone-fiber in the Seifert fibered \hyperlink{term-piece}{pieces} (Theorem \ref{factor-SF}). We need the notion of \hyperlink{term-prLen}{presentation length} for the statement and the proof of these results, so we introduce it in Subsection \ref{Subsec-prlen}.

\subsection{Presentation length}\label{Subsec-prlen}

For a finitely presented group $G$, the presentation length $\ell(G)$ of $G$ turns out to be useful in $3$-manifold topology, especially for the study of hyperbolic $3$-manifolds. For example, Daryl Cooper proved the volume of an orientable closed hyperbolic $3$-manifold $M$ is at most $\pi\ell(\pi_1(M))$ (\cite{Co}). Matthew E. White also gave an upper bound on the diameter of an orientable closed hyperbolic $3$-manifold $M$ in terms of $\ell(\pi_1(M))$ in a preprint (\cite{Wh}).

\begin{definition}\label{prlen} \raisebox{\baselineskip}[0pt]{\hypertarget{term-prLen}}Suppose $G$ is a nontrivial finitely presented group. For any presentation $\mathcal{P}=(x_1,\ldots,x_n;r_1,\ldots,r_m)$ of $G$ with the word-length of each relator
$|r_j|\geq 2$, define the \emph{length} of $\mathcal{P}$ as
$$\ell(\mathcal{P})=\sum_{j=1}^m (|r_j|-2),$$
and the \emph{presentation length} $\ell(G)$ of $G$ as the minimum of $\ell(\mathcal{P})$ among all such presentations. \end{definition}

Note by adding finitely many generators and discarding single-letter relators, $\ell(G)$ can always be realized by a \emph{triangular presentation}, namely of which the word-length of every relator equals $2$ or $3$. The length of a triangular presentation equals the number of relators of length $3$.

Associated to any finite presentation $\mathcal{P}$ of $G$, there is a \emph{presentation $2$-complex} $K$, 
which consists of a single $0$-cell $*$, and $1$-cells corresponding to the generators attached to $*$, 
and $2$-cells corresponding to the relators attached to the $1$-skeleton with respect to $\mathcal{P}$. 
The $2$-cells are all $2$-simplices or bigons if $\mathcal{P}$ is triangular.

\subsection{Drilling a short geodesic in a hyperbolic piece}\label{Subsec-factorHyp} 
In this subsection, we show that maps factorize through the \hyperlink{term-extDrilling}{extended drilling} 
of a short simple closed geodesic in a hyperbolic \hyperlink{term-piece}{piece}. 
The precise statement is as follows.

\begin{theorem}\label{factor-hyp} Let $G$ be a finitely presented group, and $M$ be an orientable aspherical compact $3$-manifold. Suppose there is a simple closed geodesic $\gamma$ in the interior of 
a hyperbolic \hyperlink{term-piece}{piece} $J\subset M$ with length sufficiently small, with respect to some complete hyperbolic metric on the interior of $J$, but depending only on the \hyperlink{term-prLen}{presentation length} $\ell(G)$. Then any homomorphism $\phi:G\to\pi_1(M)$ factors through the \hyperlink{term-extDehnFilling}{extended Dehn filling} epimorphism $\iota^e:\pi_1(N^{e(\zeta,m)})\to \pi_1(M)$ of some denominator $0<m\leq T(\ell(G))$, where $N=M-\gamma$, $\zeta$ is the meridian of the Margulis tube about $\gamma$,  and $T(n)=2\cdot 3^n$. Namely, $\phi=\iota^e\circ\phi^e$ for some $\phi^e:G\to\pi_1(N^e)$. 
\end{theorem}

\begin{remark} By Mostow Rigidity, the complete hyperbolic metric on the interior of $J$ is unique, of finite volume, if $\partial J$ consists of tori or is empty.\end{remark}

We prove Theorem \ref{factor-hyp} in the rest of this subsection. The approach here is inspired by a paper 
of Matthew E. White (\cite{Wh}).

To begin with, take a presentation $2$-complex $K$ of a triangular presentation $\mathcal{P}$ achieving $\ell(G)$, and a PL map $f:K\to M$ realizing $\phi:G\to\pi_1(M)$. We may assume $f(K)$ intersects the JSJ tori $\tori$ in general position and the image of the $0$-simplex $*\in K$ misses $J$. Let $\epsilon_3>0$ be the Margulis constant for $\HH^3$, and $J_\geo=
(\mathring{J},\rho)$ be the interior of $J$ with the complete hyperbolic metric $\rho$ as assumed. By picking a sufficiently small $\epsilon<\epsilon_3$, we may endow $M$ with a complete Riemannian metric so that $J$ is isometric to $J_\geo$ with the open $\epsilon$-thin horocusps corresponding to $\partial J$ removed.

We may homotope $f$ so that $f(K)\cap J$ is totally geodesic on each $2$-simplex of $K$, and the total area of $f(K)\cap J$ is at most $\pi\ell(\mathcal{P})$. In fact, we may first homotope $f$ so that $f(K)$ meets the JSJ tori in minimal normal position, (i.e. the number of $K^{(1)}\cap f^{-1}(\tori)$ is minimal), then homotope rel $f^{-1}(\tori)$ so that $f(K)\cap J$ becomes totally geodesic. Roughly speaking, noting that $K\cap f^{-1}(J)$ is a union of $1$-handles (bands) and monkey-handles (hexagons) by the normal position assumption, the total area of $f(K)\cap J$ is approximately the total area of the monkey-handles. The area of each monkey-handle is bounded by $\pi$, which is the area of an ideal hyperbolic triangle, and there are at most $\ell(\mathcal{P})$ monkey-handles, as $\ell(\mathcal{P})$ equals the number of $2$-simplices of $K$. It is not hard to make a rigorous argument of the estimation using elementary hyperbolic geometry, so the total area of $f(K)\cap J$ can be bounded by $\pi\ell(\mathcal{P})$.

A theorem of Chun Cao, Frederick W. Gehring and Gavin J. Martin \cite{CGM} says that if $\gamma$ has length $l<\frac{\sqrt{3}}{2\pi}(\sqrt{2}-1)$, then there is an embedded tube $V\subset M$ of radius $r$ with the core geodesic $\gamma$, such that: $$\sinh^2(r)=\frac{\sqrt{1-(4\pi l/\sqrt{3})}}{4\pi l/\sqrt{3}}-\frac{1}{2},$$ \cite{CGM}. This means if $\gamma$ is very short, it lies in a very deep tube $V$. In particular, any meridian disk of $V$ will have very large area.

Up to a small adjustment of the radius of $V$, we may assume $f(K)$ intersects $\partial V$ in general position, 
and the $0$-simplex $*\in K$ misses $V$. Denote $K_V=f^{-1}(V)$, and $K_{\partial{V}}=f^{-1}(\partial V)$. 
Let $\zeta\subset\partial V$ be an oriented simple closed curve bounding a meridian disk in $V$. 
Topologically, let $i:N\to M$ be the Dehn filling inclusion identifying $N$ as $M\setminus\mathring{V}$. 
Remember $N^e=N^{e(\zeta,m)}=N\cup Z(m)$, where $Z(m)$ is the \hyperlink{term-ridge}{ridge piece} of some denominator $m>0$. 
We must show  if $\gamma$ is sufficiently short, then for some denominator $m$, there is a map $$f^e:K\to N^e,$$ such that $f\simeq i^e\circ f^e$. 
The lemma below is an easy criterion.

\begin{lemma}\label{factorCriterion} 
Suppose $f_*:H_2(K_V,K_{\partial V};\QQ)\to H_2(V,\partial V;\QQ)$ vanishes, and
let $m >0$ be the maximal order of torsion elements of $H_1(K_V,K_{\partial V};\ZZ)$.
Then  $f^e|_{K_V}$, and hence $f^e$, exists for $N^{e(\zeta,m)}$.
\end{lemma}
\begin{proof} It suffices to show that there exists a lift $f^e|_{K_V}:K_V\to Z(m)$ commuting with the diagram up to homotopy:
$$\begin{CD} K_{\partial V} @>f_{\partial V}^e>> Z(m)\\
@V j V\cap V @V i^e VV\\
K_V@>f_V >> V,
\end{CD} $$
where $f_V, f_{\partial V}$ are the restrictions of $f:K\to M$ to $K_V, K_{\partial V}$ respectively, and $f_{\partial V}^e$ is $f_{\partial V}$ post-composed with $\partial V\subset Z(m)$. This is a relative homotopy extension problem which can be resolved by obstruction theory, but we give a manual proof here for the reader's reference. Because $\pi_1(Z(m))$ and $\pi_1(V)$ are abelian, this is the same as finding a lifing $\psi:H_1(K_V)\to H_1(Z(m))$ commuting with the diagram above on homology. If $f_*:H_2(K_V,K_{\partial V};\QQ)\to H_2(V,\partial V;\QQ)$ vanishes, so does $f_*:H_2(K_V,K_{\partial V})\to H_2(V,\partial V)$ since $H_2(V,\partial V)\cong\ZZ$. Note that since $K_V$ is homotopy equivalent to a graph, $H_2(K_V)=0$ in the relative homology sequence:
$$\cdots\to H_2(K_V)\to H_2(K_V,K_{\partial V})\to H_1(K_{\partial V})\to H_1(K_V)\to\cdots.$$ 
Thus $H_2(K_V,K_{\partial V})\leq H_1(K_{\partial V})$ is the kernel of $H_1(K_{\partial V})\to H_1(K_V)$.
From the commutative diagram:
$$\begin{CD} H_2(K_V,K_{\partial V}) @>0>> H_2(V,\partial V)\\
@VVV @VVV\\
H_1(K_{\partial V})@>>> H_1(\partial V)@>\subset>> H_1(Z(m)),
\end{CD} $$
we conclude the kernel of $H_1(K_{\partial V})\to H_1(K_V)$ is contained
in the kernel of $H_1(K_{\partial V})\to H_1(Z(m))$. Denote $A=H_1(K_V)$, $B={\rm Im}\{H_1(K_{\partial V})\to H_1(K_V)\}$. Since $A$ is a finitely generated abelian group,
$A=\bar{B}\oplus \ZZ [u_1]\oplus\cdots\oplus\ZZ [u_t]$ where $\bar{B}/B$ is torsion and $[u_i]\in A$, for $1\leq i\leq t$. Take $m>0$ to be the least common multiple of the orders of elements in $\bar{B}/B$ (equivalently, the maximal order of torsion elements in $H_1(K_V,K_{\partial V};\ZZ)$). Then $\psi:A\to H_1(Z(m))$ can be constructed as follows. 
Let $\eta\in \partial V$ be a slope intersecting the filling slope $\zeta\subset\partial V$ in one point, such that $i^e_*[\eta]=[\gamma]$. For any $[u_i]$, $1\leq i\leq t$, define $\psi([u_i])=[\tilde{u}_i]\in H_1(Z(m))$ such that $i^e_*[\tilde{u}_i]=f_{V*}[u_i]$. For any $[v]\in \bar{B}$ with order $s>0$ in $\bar{B}/B$, let $[w]\in H_1(K_{\partial{V}})$ be such that $j_*(w)=s\,[v]$. Note $f_{V*}(s\,[v])=s\,f_{V*}[v]=sb\,[\gamma]$, for some integer $b$. Then $i^e_*f_{\partial V*}[w]=i^e_*(sb\,[v])$. This means $f_{\partial V*}[w]=sb\,[\eta]+c\,[\zeta]$ for some integer $c$. Since $s$ divides $m$ by the choice of $m$, $b[\eta]+\frac{c}{s}[\zeta]\in H_1(\partial V)+\ZZ\frac{[\zeta]}{m}\cong H_1(Z(m))$. We may define $\psi([v])=b\,[\eta]+\frac{c}{s}\,[\zeta]$. It is straightforward to check $\psi$ is a well-defined homomorphism as required.
\end{proof}

For any $\RR$-coefficient chain of PL singular $2$-simplices into $M$, its area is known as the sum of the unsigned pull-back areas of the $2$-simplices, weighted by the absolute values of the coefficients. Because any PL singular relative $\ZZ$-cycle which represents a nontrivial element of $H_2(V,\partial V;\ZZ)\cong\ZZ$ will have arbitrarily large area if $V$ has sufficiently large radius, to apply the criterion in Lemma \ref{factorCriterion}, it suffices to show that $H_2(K_V,K_{\partial V};\QQ)$ has a generating set whose elements are represented by relative $\ZZ$-cycles each with area bounded in terms of $\ell(\mathcal{P})$.

\begin{lemma}\label{smallGen} There is a generating set of $H_2(K_V,K_{\partial V};\QQ)$ whose elements are represented by relative $\ZZ$-cycles each with area bounded by $A(\ell(\mathcal{P}))$, where $A(n)=27^n(9n^2+4n)\pi$.\end{lemma}

\begin{proof} Without loss of generality, we may assume $K_V$ does not contain the $0$-simplex $*$ of $K$. Because $V$ is convex and each $2$-simplex of $K$ has convex image within the hyperbolic \hyperlink{term-piece}{piece}, $K_V$ is a finite union of \emph{$0$-handles} (half-disks), \emph{$1$-handles} (bands), \emph{monkey-handles} (hexagons), and possibly a few \emph{isolated disks} (disks whose boundary do not meet the $1$-skeleton of $K$), cf. Figure \ref{figKV}. It is clear that the number of monkey-handles is at most the number of $2$-simplices, hence bounded by $\ell(\mathcal{P})$, and the union of $1$-handles in $K_V$ is an $I$-bundle over a (possible disconnected) graph. By fixing an orientation for each of them, the handles and the isolated disks give a CW-complex structure on $K_V$ in an obvious fashion. Let $\mathcal{C}_*(K_V,K_{\partial V})$, $\mathcal{Z}_*(K_V,K_{\partial V})$, $\mathcal{B}_*(K_V,K_{\partial V})$ denote the free $\ZZ$-modules of cellular relative chains, cycles and boundaries, respectively. $\mathcal{C}_2(K_V,K_{\partial V})$ has a natural basis consisting of the handles and the isolated disks. 

\begin{figure}
\centering
\psfrag{0}{\scriptsize{a $0$-handle}}
\psfrag{1}{\scriptsize{a $1$-handle}}
\psfrag{m}{\scriptsize{a monkey-handle}}
\psfrag{i}{\scriptsize{an isolated disk}}
\psfrag{K}{\scriptsize{$K_V$}}
\psfrag{P}{\scriptsize{$K_{\partial V}$}}
\includegraphics[scale=.7]{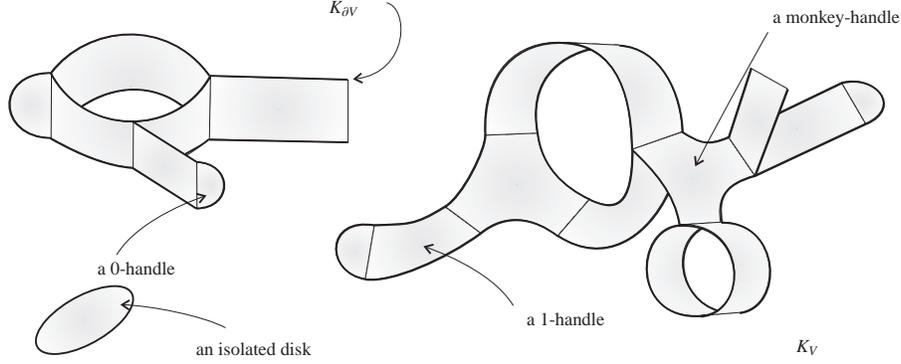}
\caption{$K_V$ and $K_{\partial V}$.}\label{figKV}
\end{figure}

To prove the lemma, it suffices to find a generating set for $\mathcal{Z}_2(K_V,K_{\partial V};\QQ)$ whose elements are in $\mathcal{Z}_2(K_V,K_{\partial V})\leq\mathcal{C}_2(K_V,K_{\partial V})$ with bounded coefficients over the natural basis. Decompose $K_V$ as: $$K_V=S_V\sqcup E_V\sqcup K'_V,$$ where $S_V$ is the union of the isolated disk components, $E_V$ is the union of the components that contain no monkey-handles, and $K'_V$ is the union of the components that contain at least one monkey-handle. Let $S_{\partial V}$, $E_{\partial V}$, $K'_{\partial V}$ be the intersection of $S_V$, $E_V$, $K'_V$ with $K_{\partial V}$, respectively. $$\mathcal{Z}_2(K_V,K_{\partial V};\QQ)=\mathcal{Z}_2(S_V,S_{\partial V};\QQ)\oplus\mathcal{Z}_2(E_V,E_{\partial V};\QQ)\oplus\mathcal{Z}_2(K'_V,K'_{\partial V};\QQ).$$
It suffices to find bounded generating relative $\ZZ$-cycles for the direct-summands separately.

First, consider $\mathcal{Z}_2(S_V,S_{\partial V};\QQ)$. Clearly, it has a generating set whose elements are the isolated disks. Hence absolute value of the coefficients over the natural basis are bounded $\leq 1$ for every element of the generating set.

Secondly, consider $\mathcal{Z}_2(E_V,E_{\partial V};\QQ)$. We show that it has a generating set whose elements have coefficients bounded $\leq2$ in absolute value over the natural basis. To see this, note $E=E_I\cup D_1\cup\cdots \cup D_s$ is a union of an $I$-bundle $E_I$ over a (possibly disconnected) graph $\Gamma_I$ together with $0$-handles $D_j$, $1\leq i\leq s$. $K_{\partial V}\cap E_I$ is an embedded $\partial I$-bundle $E_{\partial I}$. Now $\mathcal{Z}_2(E_I,E_{\partial I};\QQ)$ can be generated by all the relative $\ZZ$-cycles, in fact finitely many, of the following forms: (i) $A_I\in\mathcal{Z}_2(E_I,E_{\partial I})$, where $(A_I,A_{\partial I})\subset (E_I,E_{\partial I})$ is a sub-$I$-bundle which is an embedded annulus; or (ii) $R_I+2B_I+R'_I\in\mathcal{Z}_2(E_I,E_{\partial I})$, where $(R_I,R_{\partial I}),(R_I',R_{\partial I}')\subset (E_I,E_{\partial I})$ are sub-$I$-bundles which are embedded M\"obius strips, and $(B_I,B_{\partial I})\subset (E_I,E_{\partial I})$ is a sub-$I$-bundle which is an embedded band joining $R_I$ and $R'_I$. Moreover, $\mathcal{Z}_2(E_V,E_{\partial V};\QQ)\,/\,\mathcal{Z}_2(E_I,E_{\partial I};\QQ)$ can be generated by the residual classes represented by all the relative $\ZZ$-cycles, in fact finitely many, of the following forms: (i) $D_j+B_I\pm D_{j'}\in\mathcal{Z}_2(E_V,E_{\partial V})$, where $D_j$, $D_{j'}$ are distinct $0$-handles, and $(B_I,B_{\partial I})\subset(E_I,E_{\partial I})$ is a sub-$I$-bundle which is an embedded band joining $D_j$ and $D_{j'}$; or (ii) $2D_j+2B_I+R_I\in\mathcal{Z}_2(E_V,E_{\partial V})$, where $D_j$ is a $0$-handle, and $(R_I,R_{\partial I})\subset (E_I,E_{\partial I})$ is a sub-$I$-bundle which is an embedded M\"obius strip, and $(B_I,B_{\partial I})\subset (E_I,E_{\partial I})$ is a sub-$I$-bundle which is an embedded band joining $D_j$ and $R_I$.
All these relative $\ZZ$-cycles together generate $\mathcal{Z}_2(E_V,E_{\partial V};\QQ)$, and each of them has coefficients bounded $\leq2$ in absolute value over the natural basis.

Finally, consider $\mathcal{Z}_2(K'_V,K'_{\partial V};\QQ)$. We show that it has a generating set whose elements have coefficients bounded $\leq27^{\ell(\mathcal{P})}(9\ell(\mathcal{P})+4)$ in absolute value over the natural basis. To see this, note $K'_V=K'_I\cup D_1\cup\cdots\cup D_s\cup F_1\cdots F_t$ is a union of an $I$-bundle $K'_I$ over a (possibly disconnected) graph, and $0$-handles $D_j$, $1\leq j\leq s$, and monkey-handles $F_k$, $1\leq k\leq t$. Note $t\leq\ell(\mathcal{P})$. Moreover, $K'_I=B_1\cup\cdots\cup B_r$ is a union of $1$-handles $B_i$, $1\leq i\leq r$, and it is also a disjoint union of components $K'_{I,1},\cdots, K'_{I,p}$, where $p\leq 3t\leq 3\ell(\mathcal{P})$. 

Let $\bar\partial:\mathcal{C}_2(K'_V,K'_{\partial V})\to \mathcal{C}_1(K'_V,K'_{\partial V})$ be the relative boundary operator. Then $\mathcal{Z}_2(K'_V,K'_{\partial V};\QQ)$ is by definition the solution space of: $$\bar\partial U=0,$$ 
for $U\in\mathcal{C}_2(K'_V,K'_{\partial V};\QQ)$.
We shall first solve the residual equation $\bar\partial U=0$ modulo $\mathcal{B}_1(K'_I,K'_{\partial I})$, then lift a set of fundamental solutions to solutions of $\bar\partial U=0$ by adding chains from $\mathcal{C}_2(K'_I,K'_{\partial I})$. This set of solutions together with a generating set of $\mathcal{Z}_2(K'_I,K'_{\partial I};\QQ)$ will be a generating set of $\mathcal{Z}_2(K'_V,K'_{\partial V};\QQ)$.

To solve $\bar\partial U=0$ modulo $\mathcal{B}_1(K'_I,K'_{\partial I})$, we write:
$$U=\sum_{i=1}^r x_i\,B_i+\sum_{j=1}^s y_j\,D_j+\sum_{k=1}^t z_k\,F_k.$$
The topological interpretation of $\bar\partial U$ modulo $\mathcal{B}_1(K'_I,K'_{\partial I})$ is the total `contribution' of the base elements $B_i$, $D_j$, $F_k$'s to the fiber of each component of $K'_I$. 

To make sense of this, on each component $K'_{I,l}$ of $K'_I$, we pick an oriented fiber $\varphi_l$, $1\leq l\leq p$. Note $\mathcal{C}_1(K'_V,K'_{\partial V})=\mathcal{C}_1(K'_I,K'_{\partial I})=\mathcal{C}_1(K'_{I,1},K'_{\partial I,1})\oplus\cdots\oplus \mathcal{C}_1(K'_{I,p},K'_{\partial I,p})$, and: 
$$\mathcal{C}_1(K'_I,K'_{\partial I})\,/\,\mathcal{B}_1(K'_I,K'_{\partial I})\cong\ZZ^{\oplus p},$$ generated by $\varphi_1,\cdots, \varphi_p\bmod\mathcal{B}_1(K'_I,K'_{\partial I})$. 
The contribution of $B_i$, $D_j$, $F_k$ on $\phi_l$ is formally the value of $\bar\partial B_i$, $\bar\partial D_j$, $\bar\partial F_k$ modulo $\mathcal{B}_1(K'_I,K'_{\partial I}))$ on the $l$-th direct-summands. In other words, we count algebraically how many components of $\bar\partial B_i$ is parallel to $\varphi_l$ in $K'_{I,l}$, and similarly for $\bar\partial D_j$, $\bar\partial F_k$. In this sense, on any $\varphi_l$, each $B_i$ contributes $0$ or $\pm 2$, each $D_j$ contributes $0$ or $\pm 1$, and each $F_k$ contributes $0$, $\pm 1$, $\pm 2$ or $\pm 3$. Let $\vec{u}$ be the column vector of coordinates $(x_1,\cdots,x_r,y_1,\cdots, y_s,z_1,\cdots,z_t)^T$, and $q=r+s+t$. Let $a_{lm}$ be the contribution of the $m$-th basis vector (corresponding to some $B_i$, $D_j$ or $F_k$) on $\varphi_l$. Thus, $a_{lm}$ are integers satifying $|a_{lm}|\leq 3$, for $1\leq l\leq p$, $1\leq m\leq q$, and $\sum_{l=1}^p |a_{lm}|\leq 3$, for $1\leq m\leq q$. The residual equation $\bar\partial U=0\bmod B_1(K'_I,K'_{\partial I})$ becomes a linear system of equations: $$A\vec{u}=\vec{0},$$ where $A=(a_{lm})$ is a $p\times q$ integral matrix. Every column of $A$ has at most $3$ nonzero entries, and the sum of their absolute values is at most $3$. Our aim is to find a set of fundamental solutions over $\QQ$ with bounded integral entries.

Picking out a maximal independent collection of equations if necessary, we may assume $p$ equals the rank of $A$ over $\QQ$. We may also re-order the coordinates and assume the first $p$ columns of $A$ are linearly independent over $\QQ$. Let $A=(P,Q)$ where $P$ consists of the first $p$ columns and $Q$ of the rest $q-p$ columns. Let $\vec{u}=\left(\begin{array}{c}\vec{v}\\ \vec{w}\end{array}\right)$ be the corresponding decomposition of coordinates. Then the linear system becomes $P\vec{v}+Q\vec{w}=\vec{0}$. Basic linear algebra shows that a set of fundamental solutions
is $\vec{v}_n=-P^{-1}Q\,\vec{e}_n$, $\vec{w}_n=\vec{e}_n$, where $1\leq n\leq q-p$ and $(\vec{e}_1,\cdots,\vec{e}_{q-p})$ is the natural basis of $\RR^{q-p}$. We clear the denominator by letting ${\vec{v}_n}^*=-P^*Q\,\vec{e}_n$, ${\vec{w}_n}^*=\det(P)\,\vec{e}_n$, where $P^*$ is the adjugate matrix of $P$. The corresponding $\vec{u}^*_1,\cdots, \vec{u}^*_{q-p}$ is a set of fundamental solutions over $\QQ$ of the linear system $A\vec{u}=\vec{0}$ with integral entries. 

For each $1\leq n\leq q-p$, $\vec{u}^*_n$ has at most $p+1$ non-zero entries, and the absolute value of the entries are all bounded $3^p$. Indeed, $\vec{u}^*$ has at most $p+1$ non-zero entries by the way we picked $\vec{v}^*_n$ and $\vec{w}^*_n$. To bound the absolute value of entries, note each column of $P$ has at most $3$ nonzero entries whose absolute value sum $\leq 3$. 
It is easy to see $|\det(P)|\leq 3^p$ by an induction on $p$ using column expansions. Similarly, the absolute value of each entry of $P^*$ is at most $3^{p-1}$, and each column of $Q$ has at most $3$ nonzero entries whose absolute value sum $\leq 3$, so the absolute value of any entry of $-P^*Q$ is also $\leq 3^p$.

Let $U^*_1,\cdots, U^*_{q-p}\in\mathcal{C}_2(K'_V,K'_{\partial V})$ be the relative $2$-chains corresponding to the fundamental solutions $\vec{u}^*_1,\cdots,\vec{u}^*_{q-p}$ respectively as obtained above. Then the $U^*_n$'s form
a set of fundamental solutions to $\bar\partial U=0\bmod \mathcal{B}_1(K'_I,K'_{\partial I})$. To lift $U^*_n$
to a solution of $\bar\partial U=0$, note $\bar\partial U^*_n$ is the $\ZZ$-algebraic sum of $1$-simplices each parallel to a fiber $\varphi_l$. For a $1$-simplex $\sigma$ parallel to $\varphi_l$ coming from $\bar\partial U^*_n$, we pick a sub-$I$-bundle of $K'_{I,l}$ which is an embedded band joining $\sigma$ and $\varphi_l$, and let $L_n\in \mathcal{C}_2(K'_I,K'_{\partial I})$ be the relative $\ZZ$-chain which is the algebraic sum of all such sub-$I$-bundles.
Since each sub-$I$-bundle as a relative $\ZZ$-chain has coefficient bounded by $1$ in absolute value over the natural basis, the absolute values of coefficients of $L_n$ are bounded $\leq 3\cdot3^p(p+1)=3^{p+1}(p+1)$.
Let $\hat{U}_n=U^*_n-L_n$, $1\leq n\leq q-p$, then $\bar\partial \hat{U}_n=0$, with coefficients bounded $\leq 3^{p+1}(p+1)+3^p=3^p(3p+4)$ in absolute value.

In other words, $\hat{U}_n\in\mathcal{Z}_2(K'_V,K'_{\partial V})$, $1\leq n\leq q-p$. Moreover, $\hat{U}_n$'s together with a generating set of $\mathcal{Z}_2(K'_I,K'_{\partial I};\QQ)$ generate $\mathcal{Z}_2(K'_V,K'_{\partial V};\QQ)$. Note $K'_I$ has no monkey-handle, the no-monkey-handle case implies that $\mathcal{Z}_2(K'_I,K'_{\partial I};\QQ)$ has a generating set of relative $\ZZ$-cycles with coefficients bounded by $2$ in absolute value. Therefore, $\mathcal{Z}_2(K'_V,K'_{\partial V};\QQ)$ has a generating set of relative $\ZZ$-cycles, consisting of $\hat{U}_n$'s and the generating set of $\mathcal{Z}_2(K'_I,K'_{\partial I};\QQ)$ as above, with coefficients bounded by $3^p(3p+4)$ in absolute value. Remember $p\leq 3t\leq 3\ell(\mathcal{P})$, the absolute values of coefficients are bounded  $\leq 3^{3\ell(\mathcal{P})}(3\cdot 3\ell(\mathcal{P})+4)=27^{\ell(\mathcal{P})}(9\ell(\mathcal{P})+4)$.

To sum up, putting the generating sets of $\mathcal{Z}_2(S_V,S_{\partial V};\QQ)$, $\mathcal{Z}_2(E_V,E_{\partial V};\QQ)$, $\mathcal{Z}_2(K'_V,K'_{\partial V};\QQ)$ together, we obtain a generating set of $\mathcal{Z}_2(K_V,K_{\partial V};\QQ)$ of relative $\ZZ$-cycles with coefficients bounded by $27^{\ell(\mathcal{P})}(9\ell(\mathcal{P})+4)$ over the natural basis. In particular, they represent homology classes that generate $H_2(K_V,K_{\partial V};\QQ)$. Remember the natural basis of $\mathcal{C}_2(K_V,K_{\partial V})$ consists of handles and isolated disks, whose total area is bounded by $\pi\ell(\mathcal{P})$. Therefore, the generating set consists of relative $\ZZ$-cycles with area bounded $\leq27^{\ell(\mathcal{P})}(9\ell(\mathcal{P})+4)\cdot\Area(K_V)\leq A(\ell(\mathcal{P}))$, where $A(n)=27^n(9n^2+4n)\pi$.
\end{proof}

The following lemma bounds the torsion orders of $H_1(K_V,K_{\partial V};\ZZ)$:

\begin{lemma}\label{smallTor} The maximal order of torsion elements of $H_1(K_V,K_{\partial V};\ZZ)$ is bounded by
$T(\ell(\mathcal{P}))$, where $T(n)=2\cdot 3^n$.\end{lemma}

\begin{proof} It suffices to show that for any component $C_V$ of $K_V$, the order of torsion
elements of $H_1(C_V,C_{\partial V};\ZZ)$ is at most $T(\ell(\mathcal{P}))$, where $C_{\partial V}
=K_{\partial V}\cap C_V$. If $C_V$ is an isolated disk, $H_1(C_V,C_{\partial V};\ZZ)$ is trivial.
Thus we may assume $C_V$ is a union of $0$-handles, $1$-handles and monkey-handles.

Let $E_V$ be a maximal union of $1$-handles in $C_V$ which  forms a trivial $I$-bundle over
a (possibly disconnected) finite graph (we also include in $E_V$ isolated edges which are not contained
in any such trivial $I$-bundle, which may be thought of as trivial $I$-bundles over isolated vertices
of the finite graph). 
Suppose $E_V=E^1_V\sqcup\cdots\sqcup E^s_V$, where
each $E^j_V$ is a connected component, and let $e^j$ be a (directed) fiber of $E^j_V$.
Clearly, $H_1(E_V,E_{\partial V};\ZZ)$, where $E_{\partial V}=E_V\cap K_{\partial V}$,
is torsion-free, rank-$s$, spanned by:
$$[e^1],\cdots,[e^s].$$   
Moreover, $H_1(C_V,C_{\partial V};\ZZ)$ is generated by these $[e^j]$'s as well.
Suppose $\sigma^1,\cdots,\sigma^r$ are the rest of the handles of $C_V$, {\it i.e.} which are not
in $E_V$. The boundary of each $\sigma^i$ gives a linear combination:
$$a_{i1}\,[e^1]+\cdots +a_{is}\,[e^s]\in H_1(E_V,E_{\partial V};\ZZ).$$
Moreover, if $\sigma^i$ is a $0$-handle or $1$-handle, there is only one
non-zero coefficient which is $\pm 1$ or $\pm 2$, respectively, (the $1$-handle
case follows from the maximality of $E_V$). If $\sigma^i$ is a monkey-handle,
the absolute value of coefficients sum up
to $3$, so in particular, at most three entries are non-zero. Thus we obtain an integral $r\times s$-matrix
$A=(a_{ij})$, which is a presentation matrix of the module $H_1(C_V,C_{\partial V};\ZZ)$,
so that at most $t$ rows have more than one non-zero entry, where $t\leq\ell(\mathcal{P})$.
We may suppose these are the first $t$ rows, forming a $t\times s$-submatrix
$A'$, and the rest of the $(r-t)$ rows form a $(r-t)\times s$-submatrix $A''$.
Note the entries of $A'$ have absolute value at most $3$, so the order of any torsion elements
of ${\rm Coker}(A')$ is bounded by the greatest common divisor of the  
minors of $A'$ of square submatrices of size $\mathrm{rank}(A')$, and hence is bounded by $3^t$. As $H_1(C_V,C_{\partial V};\ZZ)$
is the quotient of $A'$ by further killing relators given by rows of $A''$, which at most doubles
the order of the  torsion, we conclude
that the orders of torsion elements of $H_1(C_V,C_{\partial V};\ZZ)$ is
at most $2\cdot 3^t$, where $t\leq \ell(\mathcal{P})$. This completes the proof.
\end{proof}

To finish the proof of Theorem \ref{factor-hyp}, if the area of the meridian disk of $V$ is larger than $A(\ell(P))$ as in Lemma \ref{smallGen}, then $f_*:H_2(K_V,K_{\partial V})\to H_2(V,\partial V)$ vanishes. 
This amounts to requiring the radius $r$ of $V$ satisfy: $$\pi\sinh^2(r)>A(\ell(\mathcal{P})).$$
If $\gamma$ is so short that this inequality holds, by Lemma \ref{factorCriterion}, we may factorize any $f:K\to M$ up to homotopy, and hence any $\phi:G\to\pi_1(M)$, through the \hyperlink{term-extDehnFilling}{extended Dehn filling}. The denominator of the drilling is bounded by the order of the torsion of $H_1(K_V,K_{\partial V})$ by Lemma \ref{factorCriterion}. By Lemma \ref{smallTor}, the order of the torsion is bounded by $T(\ell(G))$.  This completes the proof of Theorem \ref{factor-hyp}.

\subsection{Drilling a sharp cone-fiber in a Seifert fibered piece}
In this subsection, we show a similar result to Theorem \ref{factor-hyp} for Seifert fibered \hyperlink{term-piece}{pieces}, that maps factorize through the \hyperlink{term-extDrilling}{extended drilling} of an exceptional fiber at a sharp cone point in a Seifert fibered \hyperlink{term-piece}{piece}. To make this precise, we need recall some facts about Seifert fibered spaces.

Let $J$ be an orientable compact Seifert fibered space. The interior of $J$ may be regarded as an $S^1$-bundle over a finitely generated $2$-orbifold $\mathcal{O}$. In general, $\mathcal{O}$ is isomorphic to a surface with cone points and/or punctures $F(q_1,\cdots,q_s)$, where $F$ is a closed (possibly non-orientable) surface, and each integer $1<q_i\leq\infty$ ($1\leq i\leq s$) corresponds to either a cone point on $F$ with the cone angle $\frac{2\pi}{q_i}$ if $q_i<\infty$, or a puncture if $q_i=\infty$. $\mathcal{O}$ can be endowed with a complete hyperbolic structure of finite area if and only if the orbifold Euler characteristic $\chi(\mathcal{O})=\chi(F)-\sum_{i=1}^s(1-\frac{1}{q_i})$ is negative, where $\frac{1}{\infty}=0$ by convention. In other words, in this case $J$ is either $\EE\times\HH^2$-geometric or $\widetilde{\SL}_2(\RR)$-geometric. In an orientable aspherical compact $3$-manifold $M$, any Seifert fibered \hyperlink{term-piece}{piece} $J$ whose base $2$-orbifold has a sufficiently sharp cone point (i.e. the cone angle is sufficiently small) is $\EE\times\HH^2$-geometric, unless $M$ is itself $\widetilde{\SL}_2(\RR)$-geometric. In fact, if $M$ is neither $\widetilde{\SL}_2(\RR)$-geometric nor virtually solvable, any Seifert \hyperlink{term-piece}{piece} $J$ is either $\EE\times\HH^2$-geometric or homeomorphic to the nontrivial $S^1$-bundle over a M\"obius strip.

\begin{theorem}\label{factor-SF} Let $G$ be a finitely presented group, and $M$ be an orientable aspherical compact $3$-manifold. Suppose there is a sufficiently sharp cone point in the base $2$-orbifold of 
a Seifert fibered \hyperlink{term-piece}{piece} $J\subset M$, depending only on the \hyperlink{term-prLen}{presentation length} $\ell(G)$. Let $\gamma\subset J$ be the corresponding exceptional fiber, and $N=M-\gamma$ be the drilling along $\gamma$. Then any homomorphism $\phi:G\to\pi_1(M)$ factors through the \hyperlink{term-extDehnFilling}{extended Dehn filling} epimorphism $\iota^e:\pi_1(N^e)\to \pi_1(M)$ of some denominator $m\leq T(\ell(G))$. Namely, $\phi=\iota^e\circ\phi^e$ for some $\phi^e:G\to\pi_1(N^e)$.\end{theorem}

The proof is almost the same as the hyperbolic case, so we only give a sketch highlighting necessary modifications. 

We may assume $J$ is either $\EE\times\HH^2$-geometric or $\widetilde{\SL}_2(\RR)$-geometric, and let $J_\geo=(\mathring{J},\rho)$ be the interior of $J$ with a complete Riemanianian metric $\rho$ of finite volume, induced by a complete hyperbolic structure on its base $2$-orbifold $\mathcal{O}$ requiring the length of any ordinary fiber to be $1$. Let $x\in\mathcal{O}$ be the cone point as assumed with cone angle $\frac{2\pi}{q}$. A result of Gaven J. Martin implies that for any complete hyperbolic $2$-orbifold $\mathcal{O}$ with a cone point of angle $\frac{2\pi}{q}$, there is an embedded cone centered at the point with radius $r$ satisfying:
$$\cosh(r)=\frac{1}{2\sin\frac{\pi}{q}},$$
which is optimal in $S^2(2,3,q)$, (cf. \cite[Theorem 2.2]{Martin96}). Applying to $x$ as above, the preimage of the embedded cone in $J_\geo$ is a tube, which will have very large radius if the cone is very sharp.

There is a natural notion of the horizontal area of a PL singular $2$-complex in $J_\geo$, heuristically the area of its projection on the base $2$-orbifold. Formally, let $\tilde{\omega}$ be the pull-back of the area form of $\HH^2$ via the natural projection $\EE\times\HH^2\to\HH^2$ or $\widetilde{\SL}_2(\RR)\to\HH^2$, which is invariant under isometry. As it is invariant under the holonomy action of $\pi_1(J_\geo)$, $\tilde{\omega}$ descends to a $2$-form $\omega$ on $J_\geo$. For any PL singular $2$-simplex $j:\Delta\to J_\geo$, we define the \emph{horizontal area} to be: $$\Area^{\tt h}(j(\Delta))=\left|\,\int_\Delta j^*\omega\,\right|,$$ and define the horizontal area of a $\RR$-coefficient PL singular $2$-chain in $J_\geo$ to be the sum of the horizontal areas of its simplices, weighted by the absolute values of coefficients.

Because for the $\EE\times\HH^2$-geometry, resp. for the $\widetilde{\SL}_2(\RR)$-geometry, any path in $J_\geo$ can be pulled straight, namely, homotoped rel end-points to a unique geodesic segment. Moreover, any immersed $2$-simplex in $J_\geo$ can be homotoped rel vertices to a ruled $2$-simplex with geodesic sides. In fact, let $\Delta=[0,1]\times[0,1]/\sim$ where $(t,0)\sim (t',0)$ for any ($0\leq t,t'\leq 1$), and $j:\Delta\to J_\geo$ be an immersion in its interior. One may first pull straight the sides by homotopy, then simultaneously homotope so that $j(\seq{t}\times[0,1])$ becomes geodesic for every $0\leq t\leq 1$. Note every geodesic in $\EE\times\HH^2$, resp. in $\widetilde{\SL}_2(\RR)$, projects to a geodesic in $\HH^2$, it is clear that any ruled $2$-simplex (lifted) in $\EE\times\HH^2$, resp. in $\widetilde{\SL}_2(\RR)$, projects to a totally geodesic triangle in $\HH^2$. This implies any ruled $2$-simplex in $J_\geo$ has horizontal area at most $\pi$. More generally, ruled triangular $2$-complexes in $J_\geo$ with $m$ $2$-simplices have horizontal area at most $m\pi$.

To prove Theorem \ref{factor-SF}, pick a presentation $2$-complex $K$ of a triangular presentation $\mathcal{P}$ achieving $\ell(G)$, and a PL map $f:K\to M$ realizing $\phi$. 
By picking a sufficiently small $\epsilon<\epsilon_2$, we may endow $M$ with a complete Riemannian metric such that $J$ is isometric to $J_\geo$ with the $\epsilon$-thin tubes corresponding to $\partial J$ removed. We pull the part of $f^{-1}(J)$ straight, namely, homotope it rel $f^{-1}(\partial J)$ to a ruled $2$-complex. If $\epsilon$ is sufficiently small, we may assume the horizontal area of $f^{-1}(J)$ to be at most $\pi\ell(\mathcal{P})$ by the discussion above.

Suppose $\gamma$ is the singular fiber in $J$ with sufficiently small cone angle, then there is an embedded tube $V\subset J$ containing $\gamma$ with sufficiently large radius. Since $V$ is convex and the simplices meeting $V$ are ruled, it is easy to see that $K_V=f^{-1}(V)$ is a finite union of $0$-handles, $1$-handles and monkey-handles and possible a few isolated disks. The number of monkey-handles is at most the number of simplices $\ell(P)$. Let $K_{\partial V}=f^{-1}(\partial V)$. Now the horizontal area of $K_V$ is at most $\pi\ell(\mathcal{P})$.

The factorization criterion in Lemma \ref{factorCriterion} is a general fact which also applies here.
By the same argument as Lemma \ref{smallGen}, $H_2(K_V,K_{\partial V};\QQ)$ has a generating set whose elements are represented by relative $\ZZ$-cycles with horizontal area bounded by $27^{\ell(\mathcal{P})}(9\ell(\mathcal{P})+4)\cdot\Area^{\tt h}(K_V)\leq
27^{\ell(\mathcal{P})}(9\ell(\mathcal{P})^2+4\ell(\mathcal{P}))\pi$. Thus, if $\gamma$ is an exceptional fiber with the corresponding cone angle sufficiently small such that it has a tubular neighborhood $V\subset M$ of radius $r$ satisfying: $$\pi\sinh^2(r)>A(\ell(\mathcal{P})),$$ where
$A(n)=27^n(9n^2+4n)\pi$, $f_*:H_2(K_V,K_{\partial V})\to H_2(V,\partial V)$ vanishes. This implies $f:K\to M$ factors through the \hyperlink{term-extDehnFilling}{extended Dehn filling} $i^e:N^{e(\zeta,m)}\to M$ up to homotopy,  where denominator $m$ of the drilling is bounded by the order of the torsion of $H_1(K_V,K_{\partial V})$ by Lemma \ref{factorCriterion}. By Lemma \ref{smallTor}, the order of the torsion is bounded by $T(\ell(G))$. 
This completes the proof of Theorem \ref{factor-SF}.

\section{A bound of the simplicial volume}\label{Sec-volume}
In this section, we give an upper-bound of the volume of $M$ in terms of $G$, under the assumptions of Theorem \ref{main}. This gives some restrictions to the geometry of the hyperbolic \hyperlink{term-piece}{pieces} of $M$, which will be useful in Section \ref{Sec-homeoTypes}. For the purpose of certain independent interest, we prefer to prove a slightly more general result, allowing $M$ to be an compact orientable aspherical $3$-manifold with tori boundary. 

For any compact orientable manifold $M$ with tori boundary, we denote the simplicial volume of $M$ as $v_3\norm{M}$, where $v_3\approx 1.01494$ is the volume of an ideal regular hyperbolic tetrahedron and $\norm{M}$ stands for the Gromov norm. We prove the following theorem.

\begin{theorem}\label{volBound} Suppose $G$ is a finitely presented group with $b_1(G)=1$, and $M$ is an orientable compact aspherical $3$-manifold with (possibly empty) tori boundary. If $G$ maps onto $\pi_1(M)$, then: $$v_3\norm{N}\leq\pi\ell(G).$$
\end{theorem}

More generally, one may assume that $G$ is only finitely 
generated with $b_1(G)=1$ in this theorem, since any such group is the quotient of a finitely presented group $G' \twoheadrightarrow G$ with $b_1(G')=1$. 
We shall prove Theorem \ref{volBound} in the rest of this section. 

The idea is as follows. First take a finite $2$-complex $K$ realizing a triangular presentation $\mathcal{P}$ which achieves $\ell(G)$. 
Take a PL map $f:K\to M$ realizing an epimorphism $\phi:G\twoheadrightarrow\pi_1(M)$. We first show that $M-f(K)$ consists of elementary components, in the sense that the inclusion of any such component induces a homomorphism on $\pi_1$ with abelian image. By `pulling straight' $f$ in hyperbolic \hyperlink{term-piece}{pieces} of $M$ via homotopy, we may apply an isoperimetric inequality to bound the sum of their volumes by $\pi\ell(\mathcal{P})$. Then the theorem follows because $v_3\norm{M}$ is equal to the sum of the volume of hyperbolic \hyperlink{term-piece}{pieces}.

We first show $M-f(K)$ consists of elementary components. The approach we are taking here is a `drilling argument' which will also be used to prove Theorem \ref{main}.

\begin{prop}\label{elemCplment} Let $K$ be a finite $2$-complex with $b_1(K)=1$, and $M$ be an orientable compact aspherical $3$-manifold. Suppose $f:K\to M$ is a PL map (with respect to any PL structures of $K$, $M$) which induces an epimorphism on the fundamental group. Then $M-f(K)$ consists of elementary components, i.e. whose inclusion into $M$ has abelian $\pi_1$-image. \end{prop}

\begin{proof} 
This is trivial if $M$ is itself elementary. We shall assume $M$ to be non-elementary without loss of generality.

To argue by contradiction, suppose there is a non-elementary component $U$ of $M-f(K)$. We may take an embedded finite connected simplicial graph $\Gamma\subset U$ such that $\Gamma$ is non-elementary in $M$, i.e. $\pi_1(\Gamma)\to \pi_1(M)$ has non-abelian image. Let $N=M-\Gamma$. Observe that $N$ is aspherical, because if there is an embedded sphere in $N$, it bounds a ball in $M$. This ball cannot contain $\Gamma$ as $\Gamma$ is non-elementary, and hence the ball is contained in $N$. Thus, $N$ is irreducible, and therefore aspherical by the Sphere Theorem \cite{Papa}. Denote the induced map:
$$f':K\to N.$$ By picking base points of $K$ and $N$, there is an induced homomorphism $f'_\sharp:\pi_1(K)\to\pi_1(N)$. Note ${\rm Im}(f'_\sharp)\leq\pi_1(N)$ is in general of infinite index as $b_1(K)=1$ and $b_1(N)>1$. We consider the covering space $\tilde{N}$ of $N$ corresponding to ${\rm Im}(f'_\sharp)$, with the covering map:
$$\kappa:\tilde{N}\to N.$$
Assume we can prove $\chi(\tilde{N})<0$, and hence $b_1(\tilde{N})>1$ at this point, then we obtain a contradiction because $b_1(K)=1$ and $\pi_1(\tilde{N})\cong{\rm Im}(f'_\sharp)$. We shall show $\chi(\tilde{N})<0$ in a separate lemma, (Lemma \ref{notTiny}), and with that done, the proof is completed.
\end{proof}

\begin{lemma}\label{notTiny} With the assumptions in the proof of Proposition \ref{elemCplment}, $$\chi(\tilde{N})<0.$$\end{lemma}

\begin{proof}  Let $\tori$ be the JSJ tori of $N$. Note $\partial M$ is at most a torus under the assumption that $f:K\to M$ is $\pi_1$-surjective and $b_1(K)=1$. Consider the JSJ decomposition of $N$. Then the \hyperlink{term-piece}{piece} $Y$ containing the component of $\partial N$ coming from drilling $\Gamma$ is necessarily hyperbolic, and $\chi(Y)<0$. Let $Y_\Gamma\subset M$ be the union of $Y$ with the component of $M-Y$ that contains $\Gamma$. To argue by contradiction, suppose $b_1(\tilde{N})=1$.

By Proposition \ref{ScottCore} $\tilde{N}$ has an aspherical Scott core $C$
such that $C\cap\kappa^{-1}(\tori)$ are essential annuli and/or tori. Moreover, by Lemma \ref{simplifyCore}, $C$ has no non-central elementary \hyperlink{term-chunk}{chunk}, in particular, no elementary hyperbolic \hyperlink{term-chunk}{chunk}.
We claim that $C$ contains a hyperbolic \hyperlink{term-chunk}{chunk} $Q$ mapping to $Y$ under $C\subset\tilde{N}\stackrel{\kappa}\to N$.

To see this, note $f':K\to N$ factorizes as: $$K\stackrel{\tilde{f}'}{\longrightarrow} C\stackrel{\subset}{\longrightarrow} \tilde{N}
\stackrel{\kappa}{\longrightarrow} N,$$
up to homotopy. If $C$ had no hyperbolic \hyperlink{term-chunk}{chunk} mapping to $Y$, $f'$ would miss the interior of $Y_\Gamma$ up to homotopy.
Then $f:K\to M$ may be homotoped to $g:K\to M$ within  $N$ such that $g(K)$ misses the interior of $Y_\Gamma$. Clearly $\partial Y_\Gamma$ has some component which is not parallel to $\partial M$, because otherwise either $\Gamma$ or $g(K)$ is contained in a collar neighborhood of $\partial M$. This either contradicts $\Gamma$ being non-elementary, or contradicts $f\simeq g$ being $\pi_1$-surjective as $M$ is assumed to be non-elementary. Let $T$ be such a component of $\partial Y_\Gamma$. $T$ cannot be incompressible in $M$, otherwise $g(K)$ is not surjective by the Van Kampen theorem or the HNN extension. 
If $T$ is compressible, let $D\subset M$ be a compressing disk of $T$. One component of $\partial W$, where $W$ is a regular neighborhood of $D\cup T$, is a sphere $S\subset M$, which must bound a ball $B\subset M$. There are four cases:
if $D\subset Y_\Gamma$ and $B\subset Y_\Gamma$, then $Y_\Gamma$ is a solid torus containing $\Gamma$, which contradicts $\Gamma$ being non-elementary; if $D\subset Y_\Gamma$ and $B\subset (M\setminus\mathring{Y}_\Gamma)\cup W$, then $B$ contains $g(K)$, which contradicts $g$ being $\pi_1$-surjective; if $D\subset M\setminus\mathring{Y}_\Gamma$ and $B\subset M\setminus\mathring{Y}_\Gamma$, then $M\setminus\mathring{Y}_\Gamma$ is a solid torus containing $g(K)$, which contradicts the assumption that $M$ is non-elementary; if $D\subset M\setminus\mathring{Y}_\Gamma$ and $B\subset W\cup Y_\Gamma$, then $B$ contains $\Gamma$, which contradicts $\Gamma$ being non-elementary. This means $T$ cannot be compressible either. This contradiction proves the claim that $C$ must have some hyperbolic \hyperlink{term-chunk}{chunk} $Q$.

Now $Q$ is a non-elementary hyperbolic \hyperlink{term-chunk}{chunk} of $C$ by Lemma \ref{simplifyCore}. Let $\tilde{Y}$ be the component of $\kappa^{-1}(Y)$ containing $Q$. We have $\pi_1(Q)\cong\pi_1(\tilde{Y})\leq\pi_1(Y)$ is a non-elementary subgroup of $\pi_1(Y)$. Since $\chi(Y)<0$, we conclude $\chi(Q)<0$. Because $C$ cuts along annuli and tori into non-contractible \hyperlink{term-chunk}{chunks}, $$\chi(\tilde{N})=\chi(C)\leq\chi(Q)<0.$$ This implies $b_1(\tilde{N})>1$ as $H_n(\tilde{N})=0$ for $n>2$, a contradiction to the assumption that $b_1(\tilde{N})=1$.
\end{proof}

Let $J_1,\cdots, J_s$ ($s\geq0$) be the hyperbolic \hyperlink{term-piece}{pieces} in the JSJ decomposition of $M$ as assumed in Theorem \ref{volBound}. As before, we write $J_{i,\,\geo}=(\mathring{J}_i,\rho_i)$ for the interior of $J_i$ with the complete hyperbolic metric of finite volume. It is a well-known fact that only hyperbolic \hyperlink{term-piece}{pieces} contribute to the simplicial volume, namely, $$v_3\norm{M}=\sum_{i=1}^s\Vol(J_{i,\,\geo}),$$
cf. \cite[Theorem 1]{So}. Therefore, to prove Theorem \ref{volBound}, it suffices to bound the volume of hyperbolic \hyperlink{term-piece}{pieces}
of $M$, assuming $s>0$.

By picking a positive $\epsilon<\epsilon_3$, where $\epsilon_3$ is the Margulis constant for $\HH^3$, we may endow $M$ with a complete Riemannian metric so that $J_i$ is isometric to $J_{i,\,\geo}$ with the $\epsilon$-thin horocusps corresponding to $\partial J_i$ removed, (remember the JSJ \hyperlink{term-piece}{pieces} are the components of $M$ with an open regular neighborhood of the JSJ tori removed). 

Remember $K$ is a finite $2$-complex with a single base point $*$ and $2$-simplices corresponding to the relators of the triangular presentation $\mathcal{P}$. If $\epsilon>0$ is sufficiently small, we may homotope $f$ so that $*$ is not in any hyperbolic \hyperlink{term-piece}{piece}, and that $f(K)\cap J_i$ is totally geodesic on each $2$-simplex of $K$. As $K_{J_i}=f^{-1}(J_i)$ is a union of $1$-handles (bands) and at most $\ell({\mathcal{P}})$ monkey-handles (hexagons), we may bound the area of $K_{J_i}$ by $\pi m_i$, where $m_i$ is the number of monkey-handles in $K_{J_i}$, if $\epsilon>0$ is sufficiently small. Note $m_1+\cdots+m_s\leq\ell(\mathcal{P})$. Moreover, the area of $\partial J_i$ is bounded by $\frac{\epsilon^2A_i}{\epsilon_3^2}$, where $A_i$ is total area of the $\epsilon_3$-horocusp boundaries of $J_{i,\,\geo}$ corresponding to $\partial J_i$.

We need an isoperimetric inequality as below at this point (cf. \cite[Lemma 3.2]{Rafalski07}). 

\begin{lemma}\label{isopIneq} Let $Y$ be a hyperbolic $3$-manifold, and $R\subset Y$ be a connected compact PL sub-$3$-manifold. If $R$ is elementary in $Y$, then $$\Vol(R)\leq \frac{1}{2}\Area(\partial R).$$ \end{lemma}
\begin{proof} Pass to the covering of $X$ of $Y$ corresponding to the image of $\pi_1(R)\to\pi_1(Y)$, then there is a copy of $R$ in $X$ lifted from $R\subset Y$. As $R$ is elementary in $Y$, $\pi_1(X)$ is free abelian of rank $\leq2$, so we pick a $\pi_1$-injective map to a torus $f:X\to T^2$. Let $W\subset T^2$ be the union of two generator slopes on $T^2$ meeting in one point. We may assume $f$ and $f|_{\partial R}$ are transversal to both circles in $W$, then $\Sigma=f^{-1}(W)$ is a $2$-sub-complex in $X$ with finite area (since compact and measurable) such that the universal covering space $\tilde{X}$, isometric to $\HH^3$, can be constructed by gluing copies of $C_g=X\setminus\Sigma$ indexed by $g\in\pi_1(X)$. Let $\kappa:\tilde{X}\to X$, then any connected component $\tilde{R}$ of $\kappa^{-1}(R)$ is a universal covering of $R$. To illustrate, consider $\pi_1(X)\cong\ZZ\oplus\ZZ$ for instance, and let $\alpha,\beta$ be two generators such that $C_\alpha\cap C_0\neq\emptyset$, $C_\beta\cap C_0\neq\emptyset$. For any $m>0$, let $\tilde{R}_n$ be the union of all the $\tilde{R}\cap C_{i\,\alpha+j\,\beta}$, where $-m\leq i,j\leq m$. It is clear $\Vol(\tilde{R}_m)=4m^2\,\Vol(R)$, and $\Area(\partial\tilde{R}_n)=4m^2\,\Area(\partial R)+2m\,\Area(\Sigma)$. Using the isoperimetric inequality in $\HH^3$, we have $$\Vol(\tilde{R}_m)\leq \frac{1}{2}\Area(\partial\tilde{R}_m).$$
This implies $\Vol(R)\leq\frac{1}{2}\Area(\partial R)$ as $m\to+\infty$. When $\pi_1(X)$ is isomorphic to $\ZZ$ or trivial, the argument is similar.
\end{proof}

To finish the proof of Theorem \ref{volBound}, by Proposition \ref{elemCplment}, the compactification of each components of $J_i\setminus f(K)$ is also elementary. Thus, by Lemma \ref{isopIneq}, 
$\Vol(J_i)\leq \pi m_i+{\epsilon^2A_i}/{\epsilon_3^2}$.
We obtain:
$$\sum_{i=0}^s\Vol(J_i)\leq\pi\sum_{i=1}^s m_i+\frac{\epsilon^2}{\epsilon_3^2}\sum_{i=1}^sA_i\leq \pi\ell(\mathcal{P})+\frac{\epsilon^2A}{\epsilon_3^2},$$
where $A=A_1+\cdots+A_s$ is a constant independent of $\epsilon>0$. As $\epsilon\to0$, the left-hand side goes to
$v_3\norm{M}=\sum_{i=1}^s\Vol(J_{i,\,\geo})$, and the right-hand side goes to $\pi\ell(\mathcal{P})$. We conclude:
$$v_3\norm{M}\leq\pi\ell(\mathcal{P}).$$
This completes the proof of Theorem \ref{volBound}.

\section{The JSJ decomposition of knot complements}\label{Sec-cpnship}

In this section, we review the JSJ decomposition of knot complements following \cite{Bu}, and provide an equivalent data-structural description of a knot complement as a \hyperlink{term-rootedTree}{rooted tree} with vertices decorated by compatible geometric \hyperlink{term-node}{nodes}, (Proposition \ref{knotData}). This is in preparation of the proof of Theorem \ref{main}.

Let $k$ be a knot in $S^3$. For the knot complement $M=S^3-k$, i.e. $S^3$ with an open regular neighborhood of $k$ removed, the JSJ graph $\Lambda$ is a finite tree as every embedded torus in $S^3$ is separating. Moreover, $\Lambda$ has a natural rooted tree structure. \raisebox{\baselineskip}[0pt]{\hypertarget{term-rootedTree}} Recall that a finite tree is \emph{rooted} if it has a specified vertex, called the \emph{root}. The edges are naturally directed toward the root, thus every non-root vertex has a unique \emph{parent} adjacent to it. The adjacent vertices of a vertex except its parent are called its \emph{children}. Every vertex is contained in a unique \emph{complete rooted subtree}, namely the maximal subtree with the induced edge directions in which the vertex becomes the root. 

\begin{definition}\label{rootedJSJ} \raisebox{\baselineskip}[0pt]{\hypertarget{term-rootedJSJTree}} For a knot complement $M$, the associated \emph{rooted JSJ tree} $\vec\Lambda$ is a rooted tree isomorphic to the JSJ tree $\Lambda$ of $M$ with the root corresponding to the unique JSJ \hyperlink{term-piece}{piece} containing $\partial M$.
\end{definition}

The \hyperlink{term-rootedJSJTree}{rooted JSJ tree} is related to the satellite constructions of knots. In fact, for any \hyperlink{term-rootedTree}{complete rooted subtree} $\vec\Lambda_\cpn\subset\vec\Lambda$, the subspace $M_\cpn\subset M$ over $\vec\Lambda_\cpn$ is homeomorphic to the complement of a knot $k_\cpn$ in $S^3$, and the subspace $N\subset M$ over $\vec\Lambda\setminus\vec\Lambda_\cpn$ is homeomorphic to the complement of a knot $k_\pat$ in a solid torus $S^1\times D^2$ with the natural product structure. Thus $k$ is the satellite knot of $k_\cpn\subset S^3$ and $k_\pat\subset S^1\times D^2$. 

To give a more precise description of the JSJ \hyperlink{term-piece}{pieces} and how they are glued together, \cite{Bu} introduced the notion of \hyperlink{term-KGL}{KGLs}.

\begin{definition}[{\cite[Definition 4.4]{Bu}}]\label{KGL} \raisebox{\baselineskip}[0pt]{\hypertarget{term-KGL}}A \emph{knot-generating link} (KGL) is an oriented link $L=k_\pat\sqcup k_{\cpn_1}\sqcup\cdots\sqcup k_{\cpn_r}\subset S^3$, ($r\geq 0$), such that 
$k_{\cpn_1}\sqcup\cdots\sqcup k_{\cpn_r}$ is an oriented unlink.
\end{definition}

\begin{example}\label{KGL-example} \raisebox{\baselineskip}[0pt]{\hypertarget{term-KGLSF}} Figure \ref{figKGL} exhibits three families of \hyperlink{term-KGL}{KGLs}, namely, (right-handed) \emph{$r$-key-chain links} ($r>1$), \emph{$p/q$-torus knots} ($p,q$ coprime, $|p|>1$, $q>1$), and \emph{$p/q$-cable links} ($p,q$ coprime, $q>1$). Their complements are all Seifert fibered. \raisebox{\baselineskip}[0pt]{\hypertarget{term-KGLhyp}} There are also \emph{hyperbolic} KGLs, namely whose complements are hyperbolic, such as the Borromean rings with suitable assignments of components.

\begin{figure} [htb]
\centering
\psfrag{A}{\scriptsize{$r$-key-chain link, as $r=3$}}
\psfrag{B}{\scriptsize{$p/q$-torus knot, as $p/q=-2/3$}}
\psfrag{C}{\scriptsize{$p/q$-cable link, as $p/q=1/3$}}
\psfrag{1}{\scriptsize{$k_{\cpn_1}$}}
\psfrag{2}{\scriptsize{$k_{\cpn_2}$}}
\psfrag{3}{\scriptsize{$k_{\cpn_3}$}}
\psfrag{p}{\scriptsize{$k_{\pat}$}}
\includegraphics[scale=.7]{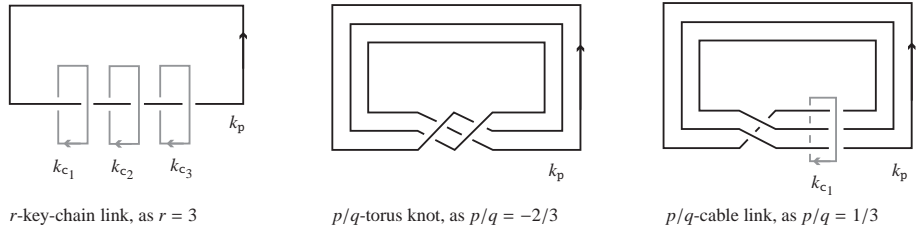}
\caption{Three families of Seifert fibered KGLs.}\label{figKGL}
\end{figure}

\end{example}

According to \cite{Bu}, the JSJ decomposition of knot complements may be described as below.

\begin{theorem}[{Cf. \cite[Theorem 4.18]{Bu}}]\label{BudneyThm}
Suppose $k$ is a nontrivial knot in $S^3$. Let $\vec\Lambda$ be the \hyperlink{term-rootedJSJTree}{rooted JSJ tree} of $M=S^3-k$. Then:

\medskip\noindent(1) Every vertex $v\in\vec\Lambda$ is associated to a \hyperlink{term-KGL}{KGL}, $L_v=k_\pat\sqcup k_{\cpn_1}\sqcup\cdots\sqcup k_{\cpn_r}\subset S^3$, satisfying the following requirements: the JSJ \hyperlink{term-piece}{piece} $J_v$ corresponding to $v$ is homeomorphic to $S^3-L_v$; $\partial J_v=\partial_\pat J_v\sqcup\partial_{\cpn_1}J_v\sqcup\cdots\sqcup\partial_{\cpn_r}J_v$, $r\geq0$, where $\partial_\pat J_v$ is the torus adjacent to the parent \hyperlink{term-piece}{piece} of $J_v$, or is $\partial M$ if $v$ is the root, and each $\partial_{\cpn_i}J_v$ is adjacent to a distinct child vertex of $v$; and if $v'$ is a child of $v$, and let $k_{\cpn'}\subset L_v$, $k'_\pat\subset L_{v'}$ be the components so that $\partial_{\cpn'} J_v$ is glued to $\partial_\pat J_{v'}$, then the meridian of $k_{\cpn'}$ is glued to the longitude of $k'_\pat\subset S^3$ preserving orientations. Note the longitude and the meridian of a component of an oriented link are naturally oriented. 

\medskip\noindent(2) There are only four possible families of \hyperlink{term-KGL}{KGLs} that could be associated to vertices of $\vec\Lambda$, namely, $r$-key-chain links ($r>1$), $p/q$-torus knots ($p,q$ coprime, $|p|>1$, $q>1$), $p/q$-cable links ($p,q$ coprime, $q>1$), and hyperbolic \hyperlink{term-KGL}{KGLs}. Furthermore, no key-chain-link vertex has a key-chain-link child. 

\medskip\noindent(3) The associated \hyperlink{term-KGL}{KGLs} are canonical 
up to unoriented isotopies of $L_v$'s with respect to the requirements. 
Moreover, any \hyperlink{term-rootedJSJTree}{rooted tree} $\vec\Lambda$ with an 
assignment of vertices to \hyperlink{term-KGL}{KGLs} satisfying the 
properties above realizes a unique nontrivial knot $k$ in $S^3$ up to isotopy.
\end{theorem}

In view of the satellite construction, the minimal complete rooted subtree of $\vec\Lambda$ 
containing a cable-link vertex corresponds to the complement of a cable knot, 
and the minimal \hyperlink{term-rootedTree}{complete rooted subtree} of $\vec\Lambda$ 
containing a key-chain-link vertex corresponds to the complement of a connected sum of knots. 

For our purpose of use, we prefer to encode a \hyperlink{term-KGL}{KGL} by its complement, forgetting the embedding into $S^3$, 
but remembering the children longitudes:

\begin{definition}\label{node} \raisebox{\baselineskip}[0pt]{\hypertarget{term-node}}
A \emph{node} is a triple $(J,\mu_\pat, \seq{\lambda_{\cpn_1},\cdots,\lambda_{\cpn_r}})$, $r\geq0$, 
where $J$ is an oriented compact $3$-manifold homeomorphic to an irreducible $(r+1)$-component 
\hyperlink{term-KGL}{KGL} complement with incompressible boundary, $\mu_\pat$ 
and $\lambda_{\cpn_i}$'s are slopes on distinct components of $\partial J$, 
i.e. oriented simple closed curves up to isotopy on $\partial J$,
such that the Dehn filling of $J$ along $\mu_\pat$ yields an $r$-component unlink complement
with meridian slopes $\seq{\lambda_{\cpn_1},\cdots,\lambda_{\cpn_r}}$ .
We call $\mu_\pat$ the \emph{parent meridian} and the  $\lambda_{\cpn_i}$'s the \emph{children longitudes}. 
It is \emph{compatible} with a vertex $v$ in a rooted tree $\vec\Lambda$ if $r$ 
equals the number of the children of $v$. Two nodes 
$(J,\mu_{\pat},\seq{\lambda_{\cpn_1},\cdots,\lambda_{\cpn_r}})$, $(J',\mu'_{\pat},\seq{\lambda'_{\cpn_1},\cdots,\lambda'_{\cpn_{r'}}})$ 
are \emph{isomorphic} if there is an orientation-preserving homeomorphism 
between the pairs $(J,\mu_{\pat}\sqcup\lambda_{\cpn_1}\sqcup\cdots\sqcup\lambda_{\cpn_r})$ 
and $(J',\mu'_{\pat}\sqcup \lambda'_{\cpn_1}\sqcup\cdots\sqcup\lambda'_{\cpn_{r'}})$, and in particular, $r=r'$.
\end{definition} 
\begin{remark}
For a node $(J,\mu_\pat, \seq{\lambda_{\cpn_1},\cdots,\lambda_{\cpn_r}})$, 
$\partial J$ is the disjoint union of components $\partial_\pat J\sqcup \partial_{\cpn_1}J\sqcup\cdots\sqcup\partial_{\cpn_r}J$, 
such that $\mu_\pat\subset \partial_\pat J$ and  $\lambda_{\cpn_i}\subset \partial_{\cpn_i}J$ for $1\leq i\leq r$. 
Each $\partial_{\cpn_i}J$ is called a \emph{child boundary}, and $\partial_\pat J$ is called the \emph{parent boundary}. 
The \emph{parent meridian} $\mu_\pat$ is determined up to finitely many possibilities by $\seq{\lambda_{\cpn_1},\cdots,\lambda_{\cpn_r}}$, and up to orientation. 
If $J$ is a key-chain, then $\mu_\pat$ is the boundary slope induced by the Seifert fibering. 
If $J$ is a torus  knot, then $\mu_\pat$ is determined by the unique meridian which makes it a knot complement. 
If $J$ is a cable link, then $\mu_\pat$ must intersect the fiber slope once if Dehn filling is to give the unknot. 
All of these possible meridians intersecting the fiber slope once 
are related by Dehn twists along the annulus connecting the two boundary components, 
and therefore $\mu_\pat$ is uniquely determined by $\lambda_{\cpn_1}$. Otherwise, if $J$ is hyperbolic, 
then there are at most $3$  possibilities for $\mu_\pat\subset \partial_\pat J$ by \cite[Theorem 2.4.4]{CGLS}. 
 There are naturally defined \emph{children meridians}, up to isotopy on $\partial J$, 
 which are the oriented simple closed curves $\mu_{\cpn_i}\subset\partial_{\cpn_i}J$ 
 such that $\mu_{\cpn_i}$ is null-homotopic in the Dehn filling of $J$ along $\mu_\pat$, 
 and that the orientation induced by $(\lambda_{\cpn_i},-\mu_{\cpn_i})$ coincides with that of $\partial_{\cpn_i}J$. 
 There is also a naturally defined \emph{parent longitude} $\lambda_\pat\subset\partial_\pat J$, up to isotopy on $\partial J$, 
 which is the oriented simple closed curve such that $\lambda_\pat$ is null-homological in the Dehn filling of $J$ along 
 $\lambda_{\cpn_1},\cdots,\lambda_{\cpn_r}$, and that the orientation induced by $(\lambda_\pat,\mu_\pat)$ coincides with that of $\partial_\pat J$.
\end{remark}

We say a \hyperlink{term-node}{node} is \emph{geometric} if $J$ is either Seifert fibered or hyperbolic. More specifically, we say \emph{key-chain nodes}, \emph{torus-knot nodes}, \emph{cable nodes} and \emph{hyperbolic nodes}, according to their defining \hyperlink{term-KGL}{KGLs}. The first three families are also called \emph{Seifert fibered nodes}.

Now Theorem \ref{BudneyThm} may be rephrased as follows.

\begin{prop}\label{knotData} Every nontrivial knot complement $M$ is completely characterized by the following data: (i) the \hyperlink{term-rootedJSJTree}{rooted JSJ tree} $\vec\Lambda$; and (ii) the assignment of the vertices of $\vec\Lambda$ to compatible geometric \hyperlink{term-node}{nodes}, each of which is either a key-chain node, or a torus-knot node, or a cable node, or a hyperbolic node.\end{prop}

Provided Proposition \ref{knotData}, in order to prove Theorem \ref{main}, we must bound the number of allowable isomorphism types of the \hyperlink{term-rootedJSJTree}{rooted JSJ tree} and the number of allowable \hyperlink{term-node}{node} types, under the assumption that $G$ maps onto the fundamental group of the knot complement $M$. This amounts to bounding the number of JSJ \hyperlink{term-piece}{pieces}, the
homeomorphism types of the JSJ \hyperlink{term-piece}{pieces}, as well as the number of allowable assignments of \hyperlink{term-node}{children longitudes}.

\section{Isomorphism types of the rooted JSJ tree}\label{Sec-number}
We start to prove Theorem \ref{main} in this section. Suppose $G$ is a finitely presented group with $b_1(G)=1$, and $M$ is a nontrivial knot complement such that there is an epimorphism $\phi:G\twoheadrightarrow\pi_1(M)$. In this section, we show that there are at most finitely many allowable isomorphism types of the \hyperlink{term-rootedJSJTree}{rooted JSJ trees} of $M$. In Section \ref{Sec-homeoTypes}, we shall show that there are at most finitely many homeomorphism types of geometric \hyperlink{term-piece}{pieces} that are allowed to be a geometric \hyperlink{term-piece}{piece} of $M$. In Section \ref{Sec-clTypes}, we shall show that there are at most finitely many allowable assignments of \hyperlink{term-node}{children longitudes} for any such \hyperlink{term-piece}{piece} to make it a \hyperlink{term-node}{node} decorating the \hyperlink{term-rootedJSJTree}{rooted JSJ tree} of $M$. By Proposition \ref{knotData}, this will complete the proof of Theorem \ref{main}.

\begin{lemma}\label{pieceNumber} Suppose $G$ is a finitely generated group of rank $n$, and $M$ is a knot complement such that $G$ maps onto $\pi_1(M)$. Then $M$ has at most $4n-3$ \hyperlink{term-piece}{pieces} in its JSJ decomposition. Hence there are at most finitely many allowable isomorphism types of \hyperlink{term-rootedJSJTree}{rooted JSJ trees}.\end{lemma}

\begin{proof}
The upper-bound of the number of geometric \hyperlink{term-piece}{pieces} is a quick consequence from a theorem of Richard Weidmann. In \cite[Theorem 2]{We}, he proved that if $G$ is a non-cyclic freely-indecomposable $n$-generated group with a minimal $k$-acylindrical action on a simplicial tree, then the graph-of-groups decomposition induced by the action has at most $1+2k(n-1)$ vertices. Recall that for a group $G$, a $G$-action on a simplicial tree $T$ is called minimal if there is no proper subtree which is $G$-invariant, and is called $k$-acylindrical if no nontrivial element of $G$ fixes a segment of length $>k$. 

Note there is a $\pi_1(M)$-action on the Bass-Serre tree $T$ associated to the JSJ decomposition of $M$. 
Precisely, $T$ is the simplicial tree constructed as follows.
Let $\tilde{M}$ be the universal covering of $M$, 
then the preimage of any geometric \hyperlink{term-piece}{piece} is a collection of component. 
A vertex of $T$ is a component of a geometric \hyperlink{term-piece}{piece} of $M$; 
two vertices are joined by an edge if and only if they are adjacent to each other. 
By specifying a base-point of $M$, there is a natural $\pi_1(M)$-action on $T$ 
induced by the covering transformation. 
Because there is no geometric \hyperlink{term-piece}{piece} homeomorphic to the nontrivial $S^1$-bundle 
over a M\"obius strip as $M$ is a knot complement, it is known that $\pi_1(M)$-action on $T$ is minimal and $2$-acylindrical, cf. \cite[p. 298]{BW}. Therefore, the induced $\phi(G)$-action on $T$ 
is also minimal and $2$-acylindrical. Since $\phi(G)$ is finitely generated as is $G$, 
we may apply Weidmann's theorem to obtain an upper-bound of the number 
of geometric \hyperlink{term-piece}{pieces} by $1+4(n-1)=4n-3$, where $n$ is the rank of $G$.

The `hence' part follows since there are only finitely many isomorphism types of \hyperlink{term-rootedTree}{rooted}  \hyperlink{term-rootedTree}{trees} with at most $4n-3$ vertices.
\end{proof}

\section{Homeomorphism types of geometric pieces}\label{Sec-homeoTypes}
In this section, we show there are at most finitely many allowable homeomorphism types of geometric \hyperlink{term-piece}{pieces} under the assumption of Theorem \ref{main}. We consider the hyperbolic case and the Seifert fibered case in Subsections \ref{Subsec-hypPieces} and \ref{Subsec-SFPieces}, respectively.

\subsection{Homeomorphism types of hyperbolic pieces}\label{Subsec-hypPieces}
In this subsection, we show there are at most finitely many allowable homeomorphism types of hyperbolic \hyperlink{term-piece}{pieces}:

\begin{prop}\label{hypPieces} Let $G$ be a finitely presented group with $b_1(G)=1$, then there are at most finitely homeomorphism types of hyperbolic \hyperlink{term-piece}{pieces} $J$ such that $J$ is a hyperbolic \hyperlink{term-piece}{piece} of 
some knot complement $M$ such that $G$ maps onto $\pi_1(M)$.
\end{prop}

We prove Proposition \ref{hypPieces} in the rest of this subsection.

Let $K$ be a finite $2$-complex $K$ of a presentation $\mathcal{P}$ of $G$ achieving $\ell(G)$.
To argue by contradiction, suppose $G$ maps onto infinitely
many knot groups $\pi_1(M_n)$ such that there are infinitely many homeomorphically distinct hyperbolic \hyperlink{term-piece}{pieces} showing up.
By Theorem \ref{volBound}, $\Vol(M_n)\leq\pi\ell(G)$. The J{\o}rgensen-Thurston
theorem (\cite[Theorem 5.12.1]{Th}) implies infinitely many of these \hyperlink{term-piece}{pieces} are distinct hyperbolic Dehn fillings 
of some hyperbolic $3$-manifold of finite volume. In particular, for any $\delta>0$
infinitely many of these \hyperlink{term-piece}{pieces} contain closed geodesics of length $<\delta$.
Let $M$ be a knot complement with a hyperbolic \hyperlink{term-piece}{piece} $J$ containing a sufficiently 
short closed geodesic $\gamma$ so that Theorem \ref{factor-hyp} holds. 
By assumption, there is a
$\pi_1$-surjective map $f:K\to M$. By Theorem \ref{factor-hyp},
it factorizes through the \hyperlink{term-extDehnFilling}{extended Dehn filling} $i^e:N^e\to M$ of some denominator $m>0$, up to homotopy,
namely $f\simeq i^e\circ f^e$, where $f^e:K\to N^e$. Remember $N=M-\gamma$ and $N^e=N\cup Z$.

Note that $b_1(K)=1$ and $b_1(N^e)=2$, so $f^e$ is not $\pi_1$-surjective. We consider the covering space
$\kappa:\tilde{N}^e\to N^e$ corresponding to ${\rm Im}(\pi_1(f^e))$, after choosing some base-points.

Let $\tori$ be union of the JSJ tori of $N^e$. 
Note $\tori$ is the union of the JSJ tori of $M$ together with the drilling boundary $\partial_\gamma N$, 
and the \hyperlink{term-ridge}{ridge piece} $Z$ of $N^e$ is only adjacent to the hyperbolic \hyperlink{term-piece}{piece} $Y=J-\gamma$. 
By Proposition \ref{ScottCore},
there is an aspherical Scott core $C$ of $\tilde{N}^e$ such
that $C\cap\kappa^{-1}(\tori)$ are essential annuli and/or tori. Moreover, because $b_1(G)=1$ implies $b_1(\tilde{N}^e)=1$, by  Lemma \ref{simplifyCore},
$C$ contains no contractible \hyperlink{term-chunk}{chunk} or non-central elementary \hyperlink{term-chunk}{chunk}, 
in particular, no elementary hyperbolic \hyperlink{term-chunk}{chunk}. 
Also, $b_0(C)=1, b_3(C)=0$ and $\chi(C)\leq 0$ (since each chunk of $C$ has $\chi\leq 0$, and the chunks are glued along tori and annuli by Lemma \ref{simplifyCore}). Therefore, since $b_1(C)=1$, we also have $b_2(C)=0$ and $\chi(C)=0$.  

However, there must be some hyperbolic \hyperlink{term-chunk}{chunk} $Q$ covering $Y$, 
because otherwise $\kappa|:C\to N^e$ would miss the interior of $Y$ up to homotopy, 
so  would map into $Z$ or $N-Y$. Either case contradicts $f:K\to M$ being $\pi_1$-surjective by the Van Kampen theorem. 
Thus $Q$ is a non-elementary hyperbolic \hyperlink{term-chunk}{chunk} of $C$. 

Suppose ${\rm Im}(\pi_1(Q)\to \pi_1(Y))$ has finite index in $\pi_1(Y)$. Then there is a torus boundary component
$\tilde{T}\subset \partial Q\cap \kappa^{-1}(\mathcal{T})$ which is adjacent to a
ridge piece. If so, then $\tilde{T}$ covers the \hyperlink{term-ridge}{side torus}  $T=\partial_\gamma Y \subset N^e$. Since $0\neq [T]\subset H_2(N^e;\QQ)$, 
we conclude that $0\neq [\tilde{T}] \in H_2(C;\QQ)$. Thus, $b_2(C) >0$, which gives a contradiction. 

We conclude $\pi_1(Q)$ is isomorphic to a non-elementary Kleinian group with infinite covolume, so $\chi(Q)<0$. Note $C$ is cut along tori/annuli into non-contractible \hyperlink{term-chunk}{chunks} with nonpositive Euler characteristics. We conclude $\chi(\tilde{N}^e)=\chi(C)\leq\chi(Q)<0$, a contradiction.

\subsection{Homeomorphism types of Seifert fibered pieces}\label{Subsec-SFPieces}
In this subsection, we show there are at most finitely many homeomorphism types of Seifert fibered \hyperlink{term-piece}{pieces}. 

\begin{prop}\label{SFPieces} Let $G$ be a finitely presented group with $b_1(G)=1$. Then there are at most finitely many homeomorphism types of Seifert fibered \hyperlink{term-piece}{pieces} $J$ such that $J$ is a Seifert fibered \hyperlink{term-piece}{piece} of 
some knot complement $M$ such that $G$ maps onto $\pi_1(M)$. In fact, there are at most finitely many allowable values of $q$ for a $p/q$-cable \hyperlink{term-piece}{piece}, where $p\neq 0, q>1$ and $p,q$ are coprime integers, and there are at most finitely many allowable values of $p,q$ for a $p/q$-torus-knot \hyperlink{term-piece}{piece}, where $|p|>1,q>1$ and $p,q$ are coprime integers.
\end{prop}

We prove Proposition \ref{SFPieces} in the rest of this subsection.

We first explain why the `in fact' part implies the first statement. Remember from Section \ref{Sec-cpnship} that there are only three families of Seifert fibered \hyperlink{term-piece}{pieces} that could be a JSJ \hyperlink{term-piece}{piece} of a nontrivial knot complement, namely the key-chain link complements, cable-link complements, and torus-knot complements, (cf. Example \ref{KGL-example}). For an $(r+1)$-component key-chain link, the homeomorphism type of its complement is determined by $r$, indeed, it is homeomorphic to  $F_{0,r+1}\times S^1$ where $F_{0,r+1}$ is $S^2$ with $(r+1)$ open disks removed. Thus the allowed values of $r$ are bounded by the number of JSJ \hyperlink{term-piece}{pieces} of $M$, which is bounded in terms of $G$ by Lemma \ref{pieceNumber}. For a $p/q$-cable-link, the homeomorphism type of its complement is determined by $q$ together with the residual class of $p\bmod q$, (possibly with some redundancy). For a $p/q$-torus knot, the homeomorphism type of its complement is determined by the value $p/q\in\QQ$, (possibly with some redudancy). Therefore, the `in fact' part and Lemma \ref{pieceNumber} implies that there are at most finitely many allowable homeomorphism types of Seifert fibered \hyperlink{term-piece}{pieces}.

It now suffices to prove the `in fact' part. The arguments for the $p/q$-cable case and the $p/q$-torus-knot case are essentially the same, but we treat them as two cases for convenience.

\bigskip\noindent\textbf{Case 1}. Homeomorphism types of cable \hyperlink{term-piece}{pieces}.

Let $K$ be a finite $2$-complex $K$ of a presentation $\mathcal{P}$ of $G$ achieving $\ell(G)$.
To argue by contradiction, suppose $G$ maps onto infinitely
many knot groups $\pi_1(M_n)$ such that there are infinitely many homeomorphically 
distinct $p_n/q_n$-cable \hyperlink{term-piece}{pieces} $J_n\subset M_n$ arising. Then for infinitely many $n$, $q_n>1$ are sufficiently large. Let $M=M_n$ be such a knot complement, $q=q_n$, $J=J_n$, and $f:K\to M$ be a $\pi_1$-surjective map as assumed. 
By Theorem \ref{factor-SF}, $f$ factors through the \hyperlink{term-extDehnFilling}{extended Dehn filling} $f^e:K\to N^e$ of some denominator $m \leq T(\ell(G))$ up to homotopy, where $N=M-\gamma$ is the drilling along the corresponding exceptional fiber $\gamma\subset J$ with boundary $T$, and $N^e=N\cup_T  Z$ is the \hyperlink{term-DehnExt}{Dehn extension}.

Consider the covering space $\kappa:\tilde{N}^e\to N^e$ corresponding to ${\rm Im}(\pi_1(f^e))$, after choosing some base-points. Then there is a homotopy lift $\tilde{f}^e:K\to \tilde{N}^e$, such that $\tilde{f}^e$ is $\pi_1$-surjective and $f^e\simeq\kappa\circ\tilde{f}^e$. Therefore $b_1(\tilde{N}^e)=1$ by the assumption $b_1(K)=1$. 

We wish to show, however, $b_1(\tilde{N}^e)>1$ in order to reach a contradiction.

Let $\tori$ be the union of the JSJ tori of $N^e$. Let $Y=J-\gamma$ be the regular cable \hyperlink{term-piece}{piece} of $N^e$, and let $Z$ be the \hyperlink{term-ridge}{ridge piece} of $N^e$ adjacent to $Y$. By Proposition \ref{ScottCore}, there is an aspherical Scott core $C$ of $\tilde{N}^e$ such that $C_{\tori}=C\cap\kappa^{-1}(\tori)$ are essential annuli and/or tori. Moreover,  because $b_1(G)=1$ implies $b_1(\tilde{N}^e)=1$, by Lemma \ref{simplifyCore} $C$ has no contractible \hyperlink{term-chunk}{chunk} and no non-central elementary \hyperlink{term-chunk}{chunk}, in particular, no Seifert fibered \hyperlink{term-chunk}{chunk} which is an $I$-bundle over an annulus.
Also, $b_0(C)=1, b_3(C)=0$ and $\chi(C)= 0$. Therefore, since $b_1(C)=1$, we also have $b_2(C)=0$.

Note $Y$ is homeomorphic to a trivial $S^1$-bundle over a pair of pants $\Sigma$, and $\partial Y$ consists of three components, namely, the parent boundary ${\partial_\pat}Y$ (the `pattern boundary'), the child boundary  ${\partial_\cpn}Y$ (the `companion boundary'), and the drilling boundary $T=\partial_\gamma Y$. We also write $\partial\Sigma={\partial_\pat}\Sigma\sqcup{\partial_\cpn}\Sigma\sqcup\partial_\gamma\Sigma$ correspondingly. Let $\tilde{Y}$ be a component of $\kappa^{-1}(Y)$ and $Q=C\cap\tilde{Y}$ be a cable \hyperlink{term-chunk}{chunk}. Then $\tilde{Y}$ is homeomorphic to either a trivial $S^1$-bundle or a trivial $\RR$-bundle over a finitely generated covering $\bar{\kappa}:\tilde{\Sigma}\to\Sigma$, and $Q$ is homeomorphic to either a trivial $S^1$-bundle or a trivial $I$-bundle over a Scott core $W$ of $\tilde\Sigma$, which is an orientable compact surface. Moreover, the \hyperlink{term-cutBdry}{cut boundary} $\partial_{\tori}Q=Q\cap\kappa^{-1}(\tori)\subset\partial Q$ is a union of annuli and/or tori. It is a sub-bundle over a corresponding union of arcs and/or loops $\partial_{\mathcal T}W\subset\partial W$, and can be decomposed as a disjoint union $${\partial_\pat}Q\sqcup{\partial_\cpn}Q\sqcup\partial_\gamma Q,$$
according to the image under $\kappa|:\partial Q\to\partial Y$. We also write $\partial_\tori W={\partial_\pat} W\sqcup{\partial_\cpn} W\sqcup\partial_\gamma W$ correspondingly. We will use the same notation for a Seifert fibered \hyperlink{term-chunk}{chunk} which is not $Y$, except that $\partial_\gamma$ will be empty in such a case, and $W$ may be an orbifold instead of a surface. 

The following lemma rules out the case that $Q$ is an $I$-bundle. 
Remembering that there is no disk component in $C_\tori$ by Lemma \ref{simplifyCore}, the lemma below says any component of $C_\tori$ is a torus unless it is adjacent to a ridge chunk.

\begin{lemma}\label{annulusRidge} If $A$ is an annulus component of $C_\tori$, then $A$ is adjacent to a ridge \hyperlink{term-chunk}{chunk}, and its core is a fiber in $Q$.\end{lemma}  

\begin{proof} Suppose $A$ were adjacent to regular \hyperlink{term-chunk}{chunks} on both sides. If $A$ is adjacent to a
hyperbolic \hyperlink{term-chunk}{chunk} $Q$, then by the argument of Proposition \ref{hypPieces}, 
$\chi(Q)<0$, and hence $\chi(C)<0$, so $b_1(C)>1$, a contradiction. 
If $A$ is adjacent to Seifert fibered \hyperlink{term-chunk}{chunks} on both sides, 
then the core loop of $A$ can only cover a fiber in one of the corresponding 
\hyperlink{term-piece}{pieces} of $N$ under $\kappa$. Then the other \hyperlink{term-chunk}{chunk} 
must be a Seifert fibered \hyperlink{term-chunk}{chunk} which is an $I$-bundle over an orientable compact surface $W$. 
Then $\partial_{\mathcal T} W$ cannot be arcs of $\partial W$ as $\partial_{\tori}Q$ has no disk component. 
Hence $\partial_{\mathcal T}W$ are a few components of $\partial W$.
Note $W$ is a compact orientable surface, which cannot be a disk since $Q$ is not contractible. 
If $\chi(W)=0$, $W$ is an annulus, so $Q$ is an $I$-bundle over an annulus, 
which has been ruled out by our simplification of the Scott core $C$, (Lemma \ref{simplifyCore}). 
We conclude $\chi(W)<0$. Note $C$ is cut into \hyperlink{term-chunk}{chunks} along annuli and tori, 
$\chi(C)\leq\chi(Q)=\chi(W)<0$, so in this case $b_1(\tilde{N}^e)=b_1(C)>1$. This contradicts $b_1(C)=1$.

If the core of $A$ is not a fiber in the cable \hyperlink{term-chunk}{chunk} $Q$, then again we conclude that
$\chi(Q)<0$, a contradiction. 
\end{proof}

Thus we may assume that every Seifert fibered \hyperlink{term-chunk}{chunk} is an $S^1$-bundle over an orientable compact surface orbifold. 

\begin{lemma} \label{ridge torus}
There is no torus $\tilde{T}$ of $C_\mathcal{T}\subset C$ adjacent to a ridge \hyperlink{term-chunk}{chunk}. 
\end{lemma}
\begin{proof}
If so, then $\tilde{T}$ covers the \hyperlink{term-ridge}{side torus}  $T=\partial_\gamma Y \subset N^e$. Since $0\neq [T]\subset H_2(N^e;\QQ)$, 
we conclude that $0\neq [\tilde{T}] \in H_2(C;\QQ)$. Thus, $b_2(C) >0$, a contradiction. 
\end{proof}

The following lemma rules out non-separating components of $C_\tori$.

\begin{lemma}\label{separating} Any component of $C_{\tori}$ is separating in $C$.\end{lemma}

\begin{proof} To argue by contradiction, suppose there were some non-separating component of $C_{\tori}$. Remember $C_\tori$ induces a graph-of-spaces decomposition of $C$ over a finite connected simplicial graph $\Lambda$. Because $b_1(C)=1$ and there is a non-separating edge, it is clear that the graph $\Lambda$ has a unique embedded loop.
Let $C^\sigma\subset C$ be the union of the \hyperlink{term-chunk}{chunk} over the embedded loop together with maximal regular \hyperlink{term-chunk}{chunks} that are adjacent to this \hyperlink{term-chunk}{chunk} at its regular \hyperlink{term-chunk}{subchunks}, and let $\sigma\subset\Lambda$ be the underlying subgraph of $C^\sigma$. Remember we do not regard ridge \hyperlink{term-chunk}{subchunks} which are homeomorphic to a $3$-manifold as regular, and hence regular \hyperlink{term-chunk}{subchunks} are all orientable. Thus the vertices of $\sigma$ corresponding to ridge \hyperlink{term-chunk}{subchunks} all lie on the embedded loop with valence $2$ in $\sigma$.

Consider the compact $3$-manifold $\hat{C}^\sigma$ obtained by replacing every ridge \hyperlink{term-chunk}{subchunk} $S\subset C^\sigma$ by a thickened annulus $A^2\times [0,1]$. This is possible because the \hyperlink{term-ridge}{ridge piece} has easily understood coverings. There is a natural `resolution' map: $$\varrho:\hat{C}^\sigma\to C^\sigma,$$ such that  $A^2\times\seq{\frac{1}{2}}$ covers the ridge of $S$, and $\varrho$ induces an isomorphism on the rational homology. In particular, $H_1(\hat{C}^\sigma;\QQ)\cong H_1(C^\sigma;\QQ)$. Note that $\hat{C}^\sigma$ may be non-orientable, but $\hat{C}^\sigma$ cut along any non-separating component of $C_\tori$ is always orientable.

However, we claim $b_1(\hat{C}^\sigma)>1$, and hence $b_1(C^\sigma)>1$. 
To see this, first note that $\hat{C}^\sigma$ is irreducible, and has no sphere boundary components by Lemma \ref{trivialCk}.
Thus, $\chi(\hat{C}^\sigma) \leq 0$. 
Suppose some non-separating component of $C_{\tori}$ is a torus $T$. 
Then $T\subset C^\sigma$, and correspondingly $T$ lies in the interior of $\hat{C}^\sigma$. 
However, $\partial \hat{C}^\sigma$ must be non-empty, since otherwise any maximal regular 
\hyperlink{term-chunk}{subchunk} $R$ of $C^\sigma$ would only have tori boundary which 
covers $\partial_\gamma N$ under $\kappa$, but this is clearly impossible as there is another 
component $\partial M\subset\partial N$. Thus, $b_3(\hat{C}^\sigma)=0$. 
Also, $0\neq [T] \in H_2(\hat{C}^\sigma;\QQ)$ since $T$ is $2$-sided, orientable, 
and non-separating, so $b_2(\hat{C}^\sigma)\geq 1$. Thus, 
$0\geq \chi(\hat{C}^\sigma) = b_0(\hat{C}^\sigma)-b_1(\hat{C}^\sigma)+b_2(\hat{C}^\sigma) \geq 2-b_1(\hat{C}^\sigma)$, which implies $b_1(\hat{C}^\sigma)>1$. 

Now suppose every non-separating component of $C_\tori$ is an annulus. 
If $A$ is such an annulus, then it must be adjacent to a ridge \hyperlink{term-chunk}{subchunk} $S\subset C^\sigma$ by Lemma \ref{annulusRidge}, so $A$ is also a component of $\partial_\gamma Q$ for the cable \hyperlink{term-chunk}{subchunk} $Q\subset C^\sigma$ adjacent to $S$ along $A$. As $\partial_\gamma Q$ is a trivial $S^1$-bundle over $\partial_\gamma W$, where $W$ is the compact orientable surface as described before, 
the core loop $\xi$ of $A$ covers an ordinary fiber in $Y$ under $\kappa$, so $\xi$ is sent to a cover of an ordinary fiber of the cable \hyperlink{term-piece}{piece} $J\subset M$ under the composition: $$\hat{C}^\sigma\stackrel{\varrho}\longrightarrow C^\sigma\stackrel{\subset}\longrightarrow \tilde{N}^e\stackrel{\kappa}\longrightarrow N^e\stackrel{i^e}\longrightarrow M.$$ Note $\partial A\subset\partial\hat{C}^\sigma$, and $A$ is $2$-sided.
If $\chi(\hat{C}^\sigma) <0$, as before we conclude that $b_1(\hat{C}^\sigma) >1$, so we may assume that $\chi(\partial \hat{C}^\sigma) = 2\chi(\hat{C}^\sigma)=0$. 
If a component of $\partial  \hat{C}^\sigma$ containing a component of $\partial A$ is a Klein bottle, then 
the fiber $\xi$ either lies in the boundary of a M\"obius strip, or is freely homotopic
to its orientation-reversal. The former case is impossible because otherwise $\hat{C}^\sigma-A$ would have at least one non-orientable boundary component, no matter the components of $\partial A$ lie on $1$ or $2$ boundary components of $\hat{C}^\sigma$, which contradicts $\hat{C}^\sigma-A$ being orientable; the latter case is impossible because $\xi$ covers an ordinary fiber of a cable \hyperlink{term-piece}{piece} of $M$, which cannot be freely homotopic to the orientation-reversal due to the $2$-acylindricity of the JSJ decomposition of $M$, cf. \cite[p. 298]{BW}. Thus $\partial A$ lies in (possibly the same) tori components of $\partial \hat{C}^\sigma$.
If $\hat{C}^\sigma$ is non-orientable with a torus boundary component, then $b_2(\hat{C}^\sigma)>0$ from the exact
sequence: 
$$0=H_3(\hat{C}^\sigma,T^2;\QQ) \to H_2(T^2;\QQ) \to H_2(\hat{C}^\sigma;\QQ),$$ so $b_1(\hat{C}^\sigma)>1$. So
we may assume that $\hat{C}^\sigma$ is orientable. First suppose $\partial A$ lies on a single component of $\partial\hat{C}^\sigma$. 
In this case, $\partial A$ must be separating in $\partial \hat{C}^\sigma$, because otherwise
the union of $A$ and a component of $\partial\hat{C}^\sigma-\partial A$ would be a Klein bottle,   
thus each component of $\partial A$ with the induced orientation would be freely homotopic to its 
orientation-reversal via the Klein bottle, but this is impossible since the components of 
$\partial A$ cover ordinary fibers in $Y\subset M$, which cannot be freely homotopic to 
their orientation-reversal in $M$ as before. Therefore, the union of $A$ and a component of $\partial\hat{C}^\sigma-\partial A$ is parallel to a non-separating $2$-sided torus in the interior of $\hat{C}^\sigma$, which implies $b_1(\hat{C}^\sigma)>1$ as before.
Now suppose $\partial A$ lies on two different components of $\partial\hat{C}^\sigma$, then $b_1(\hat{C}^\sigma)>1$ as $\hat{C}^\sigma$ is orientable with at least two tori boundary components. This proves the claim.

To finish the proof of this lemma, we successively add adjacent \hyperlink{term-chunk}{chunks} to $C^\sigma$. 
Let $C'$ be the union of $C^\sigma$ with all the adjacent ridge \hyperlink{term-chunk}{chunks}, 
then $H_1(C';\QQ)\cong H_1(C^\sigma;\QQ)$ so $b_1(C')>1$. 
For any maximal regular \hyperlink{term-chunk}{chunk} $R\subset C$ adjacent to $C'$, 
they are adjacent along a \hyperlink{term-ridge}{side annulus} of a ridge \hyperlink{term-chunk}{chunk} $S\subset C'$. 
If they are adjacent along an annulus $A$,
 $\partial R$ is non-empty, so $b_1(R\cup C')\geq b_1(R)+b_1(C')-b_1(A)\geq 1+2-1=2$.
 Thus, let $C''$ be the union of $C'$ with all the adjacent maximal regular \hyperlink{term-chunk}{chunks}, 
 $b_1(C'')>1$. Continuing in this fashion by induction, we see $b_1(C)>1$. 
 This is a contradiction since we have said $b_1(C)=1$.
\end{proof}

The last part of the proof of Lemma \ref{separating} is a useful argument, so we extract it as below.

\begin{lemma}\label{MVarg} In Case 1, assume all the components of $C_\tori$ are separating. If $C'\subset C$ is a \hyperlink{term-chunk}{chunk} with $b_1(C')>1$, and all \hyperlink{term-chunk}{chunks} adjacent to $C'$ are ridge \hyperlink{term-chunk}{chunks}, then $b_1(C)>1$.\end{lemma}
\begin{proof} Similar to the last paragraph in the proof of Lemma \ref{separating}, successively enlarge $C'$ by attaching all the adjacent ridge \hyperlink{term-chunk}{chunks}, and then all the adjacent maximal regular \hyperlink{term-chunk}{chunks}, and continue alternately in this fashion.\end{proof}

Lemmas \ref{separating} and \ref{MVarg} allow us to use Mayer-Vietoris arguments based on the homology of cable \hyperlink{term-chunk}{chunks}. This is carried out in the lemma below. Remember we have assumed any cable \hyperlink{term-chunk}{chunk} $Q$ is a trivial $S^1$-bundle over a compact orientable surface $W$.

\begin{lemma}\label{CblCkBdry} With notation as above, for any cable \hyperlink{term-chunk}{chunk} $Q\subset C$, the base surface $W$ is planar, and ${\partial_\pat}W\sqcup\partial_\gamma W$ are arcs contained in a single component of $\partial W$. Hence $W$ is homeomorphic to a regular neighborhood of the union of $\partial_\tori W$ together with an embedded simplicial tree of which each end-point lies on a distinct component of $\partial_\tori W$, connecting all the components of $\partial_\tori W$.\end{lemma}

\begin{proof} 
Suppose $W$ were not planar, there is an embedded non-separating torus $T\subset Q$ which is the sub-$S^1$-bundle over a non-separating simple closed curve of $W$. Let $C'$ be the maximal regular \hyperlink{term-chunk}{chunk} containing $W$. First suppose $\partial_\tori C'$ is empty, then $C=C'$ but $\partial C$ is non-empty with no sphere component. Thus $b_1(C)>1$ since it has non-empty aspherical boundary and a non-separating embedded torus, contrary to $b_1(\tilde{N}^e)=1$. Now suppose $\partial_\tori C'$ is non-empty, for the same reason as above, $b_1(C')>1$. By Lemmas \ref{separating}, \ref{MVarg}, $b_1(\tilde{N}^e)=b_1(C)>1$, contrary to $b_1(\tilde{N}^e)=1$. 

To see the second part, we first show ${\partial_\pat}W\sqcup\partial_\gamma W$ is contained in a single component of $\partial W$. Suppose on the contrary that ${\partial_\pat}W\sqcup\partial_\gamma W$ meets at least $2$ components of $\partial W$, then $b_1(Q)>1+k$ where $k\geq 0$ is the number of torus components of ${\partial_\cpn} Q$. Let $R\subset C$ be any maximal regular \hyperlink{term-chunk}{chunk} adjacent to $Q$ that misses the interior of $C\cap\kappa^{-1}Y$. Note if $R$ is adjacent to $Q$ along any torus component of ${\partial_\pat} Q\sqcup\partial_\gamma Q$, then $R$ has some boundary component other than this torus. Thus, let $Q'$ be $Q$ together with all such \hyperlink{term-chunk}{chunks}. A Mayer-Vietoris argument shows $b_1(Q')>1+k'$ where $k'$ is the number of torus components of ${\partial_\cpn} Q'$. Continuing in this fashion, in the end, we have $b_1(C')>1$, where $C'\subset C$ is the maximal regular \hyperlink{term-chunk}{chunk} containing $Q$, which must have no torus components of ${\partial_\cpn} C'$, (indeed, ${\partial_\pat}C'\sqcup{\partial_\cpn}C'=\emptyset$). By Lemmas \ref{separating}, \ref{MVarg}, again we obtain a contradiction. 

Now we show $\partial_\pat W\sqcup\partial_\gamma W$ are disjoint arcs rather than a single loop. Otherwise there would be two cases. If $\partial_\pat W\sqcup\partial_\gamma W=\partial_\pat W$ were a single loop, then $\partial_\pat Q$ would contain a torus mapping to the homologically non-trivial torus $\partial_\pat Y$. Thus $b_2(C)>0$, a contradiction. If $\partial_\pat W\sqcup\partial_\gamma W=\partial_\gamma W$ were a single loop, then $\partial_\gamma Q$ would contain a torus, contradicting Lemma \ref{ridge torus}.

The `hence' part is an immediate consequence from the first part.
\end{proof}

To finish the proof of Case 1, we observe the lemma below. Remember $J$ is a $p/q$-cable pattern \hyperlink{term-piece}{piece}, where $q>1$, and \textit{a priori} $M=M_\pat\cup J\cup M_\cpn$, where $M_\pat$ and $M_\cpn$ are the pattern component and the companion component of $M-J$, respectively. Recall $m$ is the denominator of the Dehn extension $N^e$. Let $q'=q/\gcd(q,m)$. 

\begin{lemma}\label{notOnto} In the current situation, $M=J\cup M_\cpn$, and the image of $(i^e\circ\kappa)_*:H_1(\tilde{N}^e)\to H_1(M)\cong\ZZ$ is contained in $q'\ZZ$.\end{lemma}

\begin{proof} We first show $M=J\cup M_\cpn$, namely $\partial_\pat J$ is parallel to $\partial M$. Suppose this is not the case, then let
$S=\partial_\pat J$ be the pattern boundary of $J$, then $S\subset N^e$. 
Notice that $0\neq [S]\in H_2(N^e;\QQ)$. If $\kappa^{-1}(S) = \emptyset$, then 
$f:K \to N$ misses $S$ up to homotopy, and we conclude that $\pi_1(K) \ntwoheadrightarrow \pi_1(M)$ unless
$J$ is the \hyperlink{term-rootedJSJTree}{root} of the JSJ tree. 
So $\kappa^{-1}(S)$ is non-empty, and by Lemma \ref{annulusRidge}, each component $\tilde{S}$ of 
$\kappa^{-1}(S)$ is a torus. But then $0\neq [\tilde{S}]\in H_2(\tilde{N}^e;\QQ)$, since $\kappa_|: H_2(\tilde{S};\QQ) \to H_2(S;\QQ)$
 is non-zero. This implies that $b_2(\tilde{N}^e) >0$, a contradiction. 

Next, we show the image of $H_1(\tilde{N}^e)\to H_1(M)\cong\ZZ$ is contained in $q'\ZZ$.
It suffices to show for $H_1(C)\to H_1(M)$ as $C\subset\tilde{N}^e$ is a homotopy equivalence. 
Now $C$ is the union of cable \hyperlink{term-chunk}{chunks}, 
ridge \hyperlink{term-chunk}{chunks} and regular $M_\cpn$-\hyperlink{term-chunk}{chunks}, 
namely the components of $C\cap\kappa^{-1}(M_\cpn)$. By the cabling construction of knots, 
the image of $H_1(M_\cpn)\to H_1(M)\cong\ZZ$ is contained in $q\ZZ$, 
so every $M_\cpn$-\hyperlink{term-chunk}{chunk} maps into $q\ZZ$ on homology. 
Note the ordinary fiber of $J$ also lies in $q\ZZ$, by Lemma \ref{CblCkBdry}, 
every cable \hyperlink{term-chunk}{chunk} also maps into $q\ZZ$ on homology. 
Note every ridge \hyperlink{term-chunk}{chunk} is adjacent to at least one cable \hyperlink{term-chunk}{chunk}, 
Lemma \ref{CblCkBdry} implies any ridge \hyperlink{term-chunk}{chunk} $S$ is homeomorphic to $\mathcal{M}_u\times I$, 
where $u:S^1\to S^1$ is a finite cyclic cover (from an ordinary Seifert fiber in $T$), and $\mathcal{M}_u$ is the mapping cylinder $u$.   
To bound the degree of
$u$, we take a basis for the homology of the torus $H_1(T)\cong \ZZ\oplus\ZZ$ with meridian $(1,0)$ and longitude $(0,1)$. Then the fiber slope is $(p,q)$. 
When we take the \hyperlink{term-DehnExt}{Dehn extension}, we embed $\ZZ + \ZZ \subset \frac{1}{m} \ZZ + \ZZ$. Letting $m'=\gcd(m,q)$, we see that the maximal root of the fiber
slope $(p,q)$ in  $\frac{1}{m} \ZZ + \ZZ$ is $(p/m', q/m')= (p/m', q')$. Thus, the degree of the map $u$ must divide $m'$ since $\mathcal{M}_u$ is a core for a cyclic cover of the ridge chunk $Z$. 
 This implies every ridge \hyperlink{term-chunk}{chunk} maps into $q'\ZZ$ on homology. Because $C\cap\kappa^{-1}(\partial Y)$, where $Y=J-\gamma$, has no non-separating component by Lemma \ref{separating}, the Mayer-Vietoris sequence implies $C$ maps into $q'\ZZ$ on homology. This completes the proof of the second part.
\end{proof}

When $q>m$, then $q'>1$, so Lemma \ref{notOnto} gives a contradiction to the assumption that $f:K\to M$ is $\pi_1$-surjective as $f$ is homotopic to the composition: $$K\stackrel{\tilde{f}^e}\longrightarrow \tilde{N}^e\stackrel{\kappa}\longrightarrow N^e\stackrel{i^e}\longrightarrow M.$$ Thus, $q\leq m\leq T(\ell(G))$ by Thoerem \ref{factor-SF}. This completes the proof of Case 1.

\bigskip\noindent\textbf{Case 2}. Homeomorphism types of torus-knot \hyperlink{term-piece}{pieces}.

We'll use the same notation as the beginning of the proof of Case 1.
In fact, the $p/q$-torus-knot complement $J$ is a Seifert fibered space over the base $2$-orbifold $S^2(p,q,\infty)$. 
If it is the $q$-exceptional fiber $\gamma$ (i.e. the fiber over the cone point correpsonding to $q$) 
that has been drilled out for the \hyperlink{term-DehnExt}{Dehn extension}, let $N^e= N \cup_T Z$
as in Case 1. Let $S=\partial_\pat J$ be the pattern boundary of $J$, then $S\subset N^e$. 
Notice that $0\neq [S]\in H_2(N^e;\QQ)$. If $\kappa^{-1}(S) = \emptyset$, then 
$f:K \to N$ misses $S$ up to homotopy, and we conclude that $\pi_1(K) \ntwoheadrightarrow \pi_1(M)$ unless
$J$ is the \hyperlink{term-rootedJSJTree}{root} of the JSJ tree. 
So $\kappa^{-1}(S)$ is non-empty, and by Lemma \ref{annulusRidge}, each component $\tilde{S}$ of 
$\kappa^{-1}(S)$ is a torus. But then $0\neq [\tilde{S}]\in H_2(\tilde{N}^e;\QQ)$, since $\kappa_|: H_2(\tilde{S};\QQ) \to H_2(S;\QQ)$
is non-zero. This implies that $b_2(\tilde{N}^e) >0$, a contradiction. 

The only possibility if $J$ is the \hyperlink{term-rootedJSJTree}{root} of the JSJ tree of $M$ is that $J=M$, i.e. $M$ is a torus knot complement. 
In this case, we obtain a contradiction by showing that covers $\tilde{N}^e\to N^e$ with $b_1(\tilde{N}^e)=1$
are elementary as in the proof of Lemma \ref{CblCkBdry}.

\bigskip Cases 1 and 2 together completes the proof of Proposition \ref{SFPieces}.

\section{Choices of parent meridians and children longitudes}\label{Sec-clTypes}
We shall finish the proof of Theorem \ref{main} in this section. 
Up to now, it remains to bound the allowable choices of \hyperlink{term-node}{children longitudes} 
on an allowable JSJ \hyperlink{term-piece}{piece}. Provided Propositions \ref{hypPieces}, \ref{SFPieces}, 
this will bound the allowable isomorphism types of \hyperlink{term-node}{nodes}.

\begin{lemma}\label{hypCl} Let $G$ be a finitely presented group with $b_1(G)=1$. 
If $J$ is an orientable compact $3$-manifold homeomorphic to a hyperbolic 
\hyperlink{term-piece}{piece} of some knot complement $M$ such that $G$ maps onto $\pi_1(M)$, 
then there are at most finitely many choices of slopes $\mu_\pat, \seq{\lambda_{\cpn_1},\cdots,\lambda_{\cpn_r}}$ 
on $\partial J$, depending only on $G$ and $J$, so that $(J,\mu_\pat, \seq{\lambda_{\cpn_1},\cdots,\lambda_{\cpn_r}})$ 
is a hyperbolic \hyperlink{term-node}{node} decorating a vertex of the \hyperlink{term-rootedJSJTree}{rooted JSJ tree} 
for some such $M$.\end{lemma}

\begin{proof}
Up to finitely many choices, we may assume the \hyperlink{term-node}{children boundaries} $\partial_{\cpn_1} J\sqcup\cdots\sqcup\partial_{\cpn_r} J$ and the \hyperlink{term-node}{parent boundary} ${\partial_\pat} J$ are assigned, (cf. the remark of Definition \ref{node}). Let $K$ be a finite presentation $2$-complex of $G$ as before. We need only show that there are finitely many possible choices for the children longitudes, since for each such choice, there are only finitely many possible \hyperlink{term-node}{parent meridian} choices. 

To argue by contradiction, suppose there are infinitely many $\pi_1$-surjective maps
$f_n:K\to M_n$ such that $J$ is a hyperbolic \hyperlink{term-piece}{piece} of a knot complement $M_n$ with the \hyperlink{term-node}{children} \hyperlink{term-node}{and parent} \hyperlink{term-node}{boundaries} compatible with the \hyperlink{term-rootedJSJTree}{rooted JSJ tree} of $M_n$, and that the sets of \hyperlink{term-node}{children} \hyperlink{term-node}{longitudes} $\seq{\lambda_{\cpn_1,n},\cdots,\lambda_{\cpn_r,n}}$ are distinct up to isotopy on $\partial J$ for different $n$'s. 
After passing to a subsequence and re-indexing the \hyperlink{term-node}{children boundaries} if necessary, we may assume that $\lambda_{\cpn_1,n}\subset\partial_{\cpn_1} J$ are distinct slopes for different $n$'s, without loss of generality .

For every $n$, we may write $M_n=M_{\pat,n}\cup J\cup M_{\cpn_1,n}\cup\cdots\cup M_{\cpn_r,n}$, where $M_{\pat,n}$ and $M_{\cpn_i,n}$'s are the components of $M_n-J$ adjacent to $J$ along $\partial_\pat J$ and $\partial_{\cpn_i}J$'s, respectively. ($M_{\pat,n}$ is possibly empty if $\partial_\pat J$ equals $\partial M_n$). Let $J'_n$ be the Dehn filling of $J$ along $\lambda_{\cpn_1,n}$, and $M'_n=M_{\pat,n}\cup J'_n\cup M_{\cpn_2,n}\cup\cdots\cup M_{\cpn_r,n}$. Then $M_n'$ is still a knot complement and there is a `de-satellitation' map:
$$\alpha_n:M_n\to M'_n,$$
namely, such that $\alpha_n|:M_{\cpn_1,n}\to S^1\times D^2\cong J'_n-J$ is the degree-one map induced by the abelianization $\pi_1(M_{\cpn_1,n})\to\ZZ$, and that $\alpha_n$ is the identity restricted to the rest of $M_n$. Since $\alpha_n$ is degree-one and hence $\pi_1$-surjective, $f'_n=\alpha_n\circ f_n:K\to M'_n$ is $\pi_1$-surjective for any $n$. However, by the J{\o}rgensen-Thurston theorem on hyperbolic Dehn fillings (\cite[Theorem 5.12.1]{Th}),  for all but finitely many slopes $\lambda_{\cpn_1,n}$, $J'_n$ are hyperbolic and mutually distinct. Thus $J'_n$'s are homeomorphically distinct hyperbolic \hyperlink{term-piece}{pieces} of $M'_n$'s, but this contradicts Proposition \ref{hypPieces} that there are only finitely many allowable homeomorphism types of hyperbolic \hyperlink{term-piece}{pieces}.
\end{proof}

\begin{lemma}\label{SFCl} 
Let $G$ be a finitely presented group with $b_1(G)=1$. If $J$ is an orientable compact $3$-manifold 
homeomorphic to a Seifert fibered \hyperlink{term-piece}{piece} of some knot complement $M$ 
such that $G$ maps onto $\pi_1(M)$, then there are at most finitely many choices of slopes 
$\mu_\pat, \seq{\lambda_{\cpn_1},\cdots,\lambda_{\cpn_r}}$ on $\partial J$, depending only on $G$ and $J$, 
so that $(J,\mu_\pat, \seq{\lambda_{\cpn_1},\cdots,\lambda_{\cpn_r}})$ is a Seifert fibered \hyperlink{term-node}{node} 
decorating a vertex of the \hyperlink{term-rootedJSJTree}{rooted JSJ tree} for some such $M$.
\end{lemma}

\begin{proof} Among the three possible families of Seifert fibered \hyperlink{term-node}{nodes} (cf. Proposition \ref{knotData}), 
the choice of \hyperlink{term-node}{children longitude} is uniquely determined up to orientation for key-chain \hyperlink{term-node}{nodes}, 
and there is no \hyperlink{term-node}{children longitude} for a torus-knot \hyperlink{term-node}{node}. 
It remains to bound the allowable choices of \hyperlink{term-node}{children longitudes} for an allowable cable \hyperlink{term-piece}{piece}. 
As before, there will be only finitely many possible choices of \hyperlink{term-node}{parent meridians} for a give choice of \hyperlink{term-node}{children longitudes}. 
Note the homeomorphism type of a $p/q$-cable \hyperlink{term-piece}{piece} ($p,q$ coprime, $q>1$) 
is determined by $q$ and the residual class of $p\bmod q$, and the \hyperlink{term-node}{children longitude} 
is determined by the integer $p$ provided $q$.
Thus it suffices to bound the allowable values of $p$ as the allowable values of $q$ are already bounded by Proposition \ref{SFPieces}.

To see this, let $G$ is a finitely presented group with $b_1(G)=1$, represented by a finite presentation $2$-complex $K$.
Let $J$ be an orientable compact $3$-manifold homeomorphic to a cable \hyperlink{term-piece}{piece} of some knot complement $M$ as assumed. There is a unique choice of the parent (i.e. pattern) component and the child (i.e. companion) component, $\partial J=\partial_\cpn J\sqcup\partial_\pat J$. Let $M=M_\pat\cup J\cup M_\cpn$ where $M_\pat$, $M_\cpn$ are the components of $M-J$ adjacent to $J$ along $\partial_\pat J$, $\partial_\cpn J$, respectively. If $J$ is realized as a $p/q$-cable in $M$, let $\lambda_\cpn\subset\partial_\cpn J$ be the \hyperlink{term-node}{child longitude} realized in $M$, and $J'$ be the Dehn filling of $J$ along $\lambda_\cpn $, and $M'=M_\pat\cup J'$. Then $M'$ is still a knot complement and there is a `de-satellitation' map:
$$\alpha:M\to M',$$
induced by the abelianization $\pi_1(M_\cpn)\to\ZZ$. Because $\alpha$ is degree-one and hence $\pi_1$-surjective, $f'=\alpha\circ f:K\to M'$ is also $\pi_1$-surjective. However, now $J'$ is a $p/q$-torus-knot complement, so there are at most finitely many allowable values of $p$ by Proposition \ref{SFPieces}.
\end{proof}

\begin{proof}[{Proof of Theorem \ref{main}}]
First, we reduce from the case that $G$ is a finitely generated group with $b_1(G)=1$ to the case that
$G$ is finitely presented.  For any presentation $\mathcal{P} = (x_1,\ldots,x_n;r_1,\ldots,r_m, \ldots)$ of $G$, we may choose a finite
collection of relators $\{r_1,\ldots, r_k\}$ such that the group $G' = \langle x_1, \ldots,  x_n;  r_1, \ldots,  r_k\rangle$ has $b_1(G')=1$. 
If $G'$ has only finitely many homomorphisms to knot groups, then so does $G$, since we have an epimorphism $G'\to G$. 
We thank Jack Button for pointing out this observation to us. Thus, we may assume that $G$ is finitely presented. 

By Lemma \ref{pieceNumber}, there are at most finitely many allowable isomorphism types of the \hyperlink{term-rootedJSJTree}{rooted JSJ tree} of $M$ (Definition \ref{rootedJSJ}) under the assumption in the statement. By Propositions \ref{hypPieces}, \ref{SFPieces}, there are at most finitely many allowable homeomorphism types of JSJ \hyperlink{term-piece}{pieces}, and by Lemmas \ref{hypCl}, \ref{SFCl}, each of them allows at most finitely many choices of \hyperlink{term-node}{children longitudes} and \hyperlink{term-node}{parent meridian}. Hence there are at most finitely many allowable isomorphism types of compatible geometric \hyperlink{term-node}{nodes}. By Proposition \ref{knotData}, we conclude that there are at most finitely many allowable homeomorphisms types of knot complements $M$ as assumed.
\end{proof}

\section{A diameter bound for closed hyperbolic 3-manifolds} \label{Sec-diamBound}

In this section, we generalize and improve the diameter bound for closed hyperbolic $3$-manifolds obtained in \cite{Wh}.

\begin{theorem}\label{diamBound} There exists a universal constant $C>0$, 
such that for any orientable closed hyperbolic 3-manifold $M$, the following statements are true.

\medskip\noindent(1) The diameter of $M$ is bounded by $C\,\ell(G)$ for any finitely presented group $G$ with $b_1(G)=0$ if $G$ maps onto $\pi_1(M)$.

\medskip\noindent(2) The diameter of $M$ is also bounded by $C\,\ell(G)$ for $G=\pi_1(M)$.\end{theorem}

%\begin{remark}  (this remark is nonsense - an infinitely presented group has infinite l(G))
%More generally in part (1), one may assume that $G$ is only finitely 
%generated with $b_1(G)=0$ in this theorem, since any such group is the quotient of a finitely presented group $G' \twoheadrightarrow G$ with $b_1(G')=0$.
%\end{remark}
 
\begin{proof} (1) Let $\epsilon=\epsilon_3>0$ be the Margulis constant of $\HH^3$. 
Let $M=M_\epsilon\cup V_1\cup\cdots\cup V_s$ where $M_\epsilon$ is the $\epsilon$-thick part, 
and $V_i$'s are the components of the $\epsilon$-thin parts, which are homeomorphic to solid tori. 

To bound the diameter of $M_\epsilon$, pick a maximal collection of pairwise disjoint balls of radii $\frac{\epsilon}{2}$. 
Then the centers of the balls form an $\epsilon$-net of $M_\epsilon$. 
In particular, $\diam(M_\epsilon)$ is bounded by $2\epsilon$ times the number of balls. 
On the other hand, writing $\omega$ for $\pi(\sinh(\epsilon)-\epsilon)$, 
i.e. the volume of a hyperbolic ball of radius $\frac{\epsilon}{2}$, 
the number of balls is at most $\frac{\Vol(M_\epsilon)}{\omega}\leq\frac{\Vol(M)}{\omega}$, 
which is bounded by $\frac{\pi\ell(G)}{\omega}$ by Theorem \ref{volBound}. 
We have:
$$\diam(M_\epsilon)\leq \frac{2\pi\epsilon\ell(G)}{\omega}=C_1\,\ell(G),$$
where $C_1=\frac{2\pi\epsilon}{\omega}$.

To bound the diameter of the thin tubes, let $V=V_i$ be a component. 
The core loop of $V$ is a simple closed geodesic $\gamma$. 
If the length of $\gamma$ were so short that the tube radius of $V$ satisfies $\pi\sinh^2(r)>A(\ell(G))$, 
where $A(n)=27^n(9n^2+4n)\pi$, (cf. Lemma \ref{smallGen}), 
then by Theorem \ref{factor-hyp}, the assumed epimorphism $\phi:G\to \pi_1(M)$ 
would factorize through some \hyperlink{term-extDrilling}{extended drilling} $\pi_1(N^e)$ 
where $N=M-\gamma$, as $\phi=\iota^e\circ\phi^e$, 
where $\iota^e:\pi_1(N^e)\to M$ is the \hyperlink{term-extDehnFilling}{extended Dehn filling} epimorphism and $\phi^e:G\to\pi_1(N^e)$. 
Consider the covering $\kappa:\tilde{N}^e\to N^e$ corresponding to ${\rm Im}(\phi^e)$, 
and an argument similar to Proposition \ref{hypPieces} would give a contradiction. 
Note assuming $b_1(G)=0$ is necessary here since after drilling one can only conclude $b_1(N^e)>0$. 
Indeed, assuming $b_1(G)=1$ would not work, for example, 
any one-cusped hyperbolic $3$-manifold maps $\pi_1$-onto infinitely many Dehn fillings whose diameters can be arbitrarily large. 
The contradiction implies the tube radius $r$ satisfies: 
\begin{eqnarray*}
r&\leq& {\rm arcsinh} \sqrt{27^{\ell(G)}(9\ell(G)^2+4\ell(G))}\\
&<&\ln \left[(1+\sqrt{2})\cdot\sqrt{27^{\ell(G)}(9\ell(G)^2+4\ell(G))}\right]\\
&<&C_2\,\ell(G),
\end{eqnarray*}
for some constant $C_2>0$. We have $\diam(V)<{\rm Length}(\gamma)+2r<\epsilon+2C_2\,\ell(G)$.

Combining the bounds for the different parts, we have 
$$\diam(M)\leq\diam(M_\epsilon)+2\,\max_{1\leq i\leq s}\,\diam(V_i)<2\epsilon+(C_1+2C_2)\ell(G)<C\,\ell(G),$$ 
for some sufficiently large universal constant $C>0$.

(2) The proof of (1) works for this situation after two modifications. 
First, in bounding the diameter of the thick part, we use \cite[Theorem 0.1]{Co} instead of Theorem \ref{volBound} to conclude $\Vol(M)\leq\pi\ell(G)$. 
Secondly, in bounding the diameter of the thin tubes, after the factorization $\phi=\iota^e\circ\phi^e$, 
we consider the sequence of homomorphisms: 
$$G\stackrel{\tilde{\phi}^e}\longrightarrow\pi_1(\tilde{N}^e)\stackrel{\kappa_\sharp}\longrightarrow\pi_1(N^e)\stackrel{\iota^e}\longrightarrow\pi_1(M).$$
The composition is $\pi_1$-isomorphic, so should that on $H_3(-;\QQ)$. Note $H_3(G;\QQ)\cong H_3(M;\QQ)\cong\QQ$, 
but $H_3(N^e;\QQ)\cong H_3(N;\QQ)=0$, cf. Subsection \ref{Subsec-DehnExt}. 
This is a contradiction, which in turn implies the tube radius $r$ satisfies $\pi\sinh^2(r)\leq A(\ell(G))$. The rest of the proof proceeds the same way as in (1).
\end{proof}

\section{Conclusion}

We believe that some of the techniques in this paper may have applications to questions regarding 
homomorphisms from groups to 3-manifold groups. 

\begin{question} \label{homomorphism}
For a finitely generated group $G$, is there a uniform description of all homomorphisms from $G$ to 
all 3-manifold groups? 
\end{question}
If $G$ is a free group of rank $n$, then this question is asking for a description of all 3-manifolds which have
a subgroup of rank $n$, which is probably too difficult to carry out in general. So to make progress on this
question, one would likely have to make certain restrictions on the group $G$, such as the hypothesis
in this paper of $b_1(G)=1$. The answer will likely involve factorizations through \hyperlink{term-DehnExt}{Dehn extensions}, and
may be analogous to the theory of limit groups and Makanin-Razborov diagrams \cite{Sela01}. 

More specifically, we ask for an effective version of Theorem \ref{main}:
\begin{question}
For a finitely presented group $G$ with $b_1(G)=1$, is there an algorithmic description
of all knot complements $M$
for which there is an epimorphism $G\twoheadrightarrow \pi_1(M)$, and for each
such $M$, an algorithmic description of the epimorphisms? 
\end{question}

As an aspect of Question \ref{homomorphism}, we ask:
\begin{question}
If $G$ is finitely generated (but infinitely presented), is there a finitely presented group $\hat{G}$ and an epimorphism $e: \hat{G}\to G$
such that it induces a bijection $e^*:{\rm Hom}(G,\Gamma)\to {\rm Hom}(\hat{G},\Gamma)$, for every $3$-manifold group $\Gamma=\pi_1(M)$?
\end{question}
This is true if we restrict $M$ to be hyperbolic. 

The proof of Theorem \ref{main} holds if we restrict $M$ to be a  hyperbolic knot complement in a rational homology sphere. 
The place that we used that $M$ is a knot complement in $S^3$ is in the JSJ decomposition in  
Section \ref{Subsec-SFPieces} and in bounding the companions in Section \ref{Sec-clTypes}. 
\begin{question}
If $G$ is finitely generated with $b_1(G)=1$, are there only finitely many $M$ a knot complement in a rational homology sphere
for which there is an epimorphism $G \twoheadrightarrow \pi_1(M)$? 
\end{question}

\bibliographystyle{hamsplain}
\bibliography{03-31-2011_AL-Simon_conjecture_submit}
\end{document}